\documentclass[12pt,leqno,a4paper]{amsart}

\title{Gulbrandsen--Halle--Hulek degeneration and Hilbert--Chow morphism}
\author{Yasunari Nagai}

\address{
Department of Mathematics, Faculty of Science and Engineering,
Waseda University,
3-4-1 Ohkubo, Shinjuku, Tokyo 169-8555, Japan}
\email{nagai.y@waseda.jp}

\usepackage{amsthm}
\usepackage{amssymb}
\usepackage{amsxtra}
\usepackage{amsmath}
\usepackage[mathscr]{eucal}
\usepackage{enumerate}
\usepackage[initials,alphabetic,lite]{amsrefs}
\usepackage{arydshln}
\usepackage[all]{xy}
\usepackage{graphicx}

\usepackage{times,mathptmx,bm}

\theoremstyle{plain}
\newtheorem{theorem}{Theorem}[subsection]
\newtheorem{lemma}[theorem]{Lemma}
\newtheorem{proposition}[theorem]{Proposition}

\newtheorem{assumption}[theorem]{Assumption}
\newtheorem*{theorem*}{Theorem}
\newtheorem*{corollary*}{Corollary}
\newtheorem*{MainTheorem}{Main Theorem}
\theoremstyle{definition}

\newtheorem{example}[theorem]{Example}

\theoremstyle{remark}

\newtheorem{remark}[theorem]{Remark}

\DeclareSymbolFont{cmletters}{OML}{cmm}{m}{it}
\DeclareSymbolFont{cmsymbols}{OMS}{cmsy}{m}{n}
\DeclareSymbolFont{cmlargesymbols}{OMX}{cmex}{m}{n}
\DeclareMathSymbol{\myjmath}{\mathord}{cmletters}{"7C}
\DeclareMathSymbol{\myamalg}{\mathbin}{cmsymbols}{"71}
\DeclareMathSymbol{\mycoprod}{\mathop}{cmlargesymbols}{"60}
\let\jmath\myjmath

\let\coprod\mycoprod

\def\lto{\longrightarrow}

\DeclareMathOperator{\Hilb}{Hilb}
\DeclareMathOperator{\Sym}{Sym}

\DeclareMathOperator{\Proj}{Proj}

\DeclareMathOperator{\pr}{pr}

\DeclareMathOperator{\id}{id}

\DeclareMathOperator{\Ker}{Ker}

\DeclareMathOperator{\Sing}{Sing}

\DeclareMathOperator{\Stab}{Stab}

\def\git{/\!\!/}
\newcommand{\Vspc}{\vphantom{\vdots}}

\makeatletter
\def\@seccntformat#1{%
  \protect\textup{\protect\@secnumfont
    \ifnum\pdfstrcmp{subsection}{#1}=0 \bfseries\fi
    \csname the#1\endcsname
    \protect\@secnumpunct
  }%
}
\makeatother

\makeatletter
   
   \@addtoreset{equation}{section}
\makeatother

\begin{document}

\baselineskip 17pt
\parskip 5pt

\begin{abstract}
For a semistable degeneration of surfaces without a triple point,
we show that two models of degeneration of Hilbert scheme of points of the family,
Gulbrandsen-Halle-Hulek degeneration given in \cite{GHH} and
the one given by the author in \cite{N}, are actually
isomorphic.
\end{abstract}

\maketitle
\vspace{-1cm}

\section*{Introduction}

The Hilbert scheme of points on a surface appears as an interesting object
in many branches of mathematics, such as holomorphic symplectic geometry, differential
geometry, singularity theory, representation theory, and so on.
If one wants to study a moduli behavior of Hilbert scheme of points on surfaces,
it is natural to ask for a good model of degenerating family of Hilbert schemes.

For the first sight, one might regard this question a triviality;
for a semistable degeneration $S\to C$ of quasi-projective surfaces,
shouldn't the relative Hilbert scheme of points $\Hilb ^n(S/C)\to C$ do the work?
However, even though the family satisfies several good properties such as
unipotency of monodromy operators on the cohomology groups,
the singular fiber of $\Hilb ^n(S/C)\to C$ can be quite singular.
In fact, it is not quite clear how to cut out the `main component' of
the relative Hilbert scheme. Moreover, even in the case of $n=2$, the relative Hilbert
scheme is not a minimal model in the sense of
higher dimensional birational geometry \cite{N08}.
Therefore, a search for a minimal model that is very near to being semistable
as a family over the base curve $C$ is
a non-trivial problem.

At the time of writing, there are at least two approaches to the problem.
One is an approach of Gulbrandsen-Halle-Hulek \cite{GHH} based on the notion of
\emph{expanded degeneration} due to Jun Li. They associate to the family $S\to C$
a family of expanded degeneration $S[n]\to \mathbb A^{n+1}$,
consider the relative Hilbert scheme $\Hilb ^n(S[n]/\mathbb A^{n+1})$ of
the expanded degeneration,
and define $I^n_{S/C}$ to be a GIT quotient $\Hilb ^n(S[n]/\mathbb A^{n+1})^{ss}\git G[n]$
for a natural action of $G[n]\cong \mathbb (\mathbb C^*)^n$ with a certain linearization.
We call the family $I^n_{S/C}\to C$ \emph{Gulbrandsen-Halle-Hulek degeneration}
(GHH degeneration as a shorthand). The other construction is in the previous work of the author \cite{N};
it works in a local situation that $S=\mathbb A^3 \to C=\mathbb A^1$ is given by
$(x,y,z)\mapsto t=xy$, and
analyzes the local structure of the singularities of
the relative symmetric product $\Sym ^n(S/C)$. We construct a $\mathbb Q$-factorial terminalization
$Y^{(n)}\to \Sym ^n(S/X)$ explicitly; first we consider a small projective toric resolution
$\tilde Z^{(n)}$ of the relative self-product $(S/C)^n$ and define $Z^{(n)}=\tilde Z^{(n)}/\mathfrak S_n$.
$Y^{(n)}$ is given as a crepant divisorial partial resolution of $Z^{(n)}$.

Each approach has its own merit; the construction of GHH degeneration is global in nature.
Gulbrandsen et. al. clarified the necessary and sufficient condition that the GHH degeneration
$I^n_{S/C}$ be projective over the base and analyzes the combinatorial properties of the degenerate
fiber. On the other hand, the approach of \cite{N} clarifies the local singularities along the singular fiber
in every step of the construction of the minimal model.

Now, another natural question is to ask the relationship between these two models.
The main theorem of this article is the following:

\begin{MainTheorem}[=Theorem \ref{comparison theorem}]
  GHH degeneration $I^n_{S/C}$ and $Y^{(n)}$ are isomorphic to each other
  as a family over $C$.
\end{MainTheorem}

The main device to prove the theorem is \emph{Hilbert-Chow morphism};
we have a natural relative Hilbert-Chow morphism
\[
\Hilb ^n(S[n]/\mathbb A^{n+1})^{ss}\to \Sym ^n(S[n]/\mathbb A^{n+1}),
\]
which is $G[n]$-equivariant. Taking GIT quotient by $G[n]$,
we get a birational morphism
\[
I^n_{S/C}=\Hilb ^n(S[n]/\mathbb A^{n+1})^{ss}\git G[n]
\to \Sym ^n(S[n]/\mathbb A^{n+1})^{ss}\git G[n].
\]
The main technical claim is that the quotient stack $[\Sym ^n(S[n]/\mathbb A^{n+1})^{ss}/G[n]]$
is isomorphic to $[\tilde Z^{(n)}/\mathfrak S_n]$ (Theorem \ref{thm: comparison of stabilizers}).
We prove this claim relying on toric geometry, in particular a description of torus quotient of
a semi-projective toric variety via polyhedron. In the process, we will see that
the choice of linearization that Gulbrandsen--Halle--Hulek made (we call it GHH linearization)
is very natural also with a view toward toric--combinatorial aspect of the theory.

\section{Toric blowing-up and its GIT quotient}

\subsection{Toric variety via polyhedron}
First we review the description of a semi-projective toric varieties using
lattice polyhedra. For details, we refer \cite{CLS}, Chapter 7.

Let $T=(\mathbb C^*)^n$ be a torus, $M=\mathbb Z^n$ a character lattice of $T$,
and $N=M^{\vee}$ a lattice of one-parameter subgroups of $T$.
Let $\tilde P$ be a \emph{lattice polyhedron} on $M$ (\emph{op. cit.} Definition 7.1.3).
Then, $\tilde P$ is a Minkowski sum of a lattice polytope $P$ and a strongly convex
rational polyhedral cone $C$,
the \emph{recession cone} of $\tilde P$.
The normal fan $\Sigma _{\tilde P}$ defines a toric variety $X_T(\tilde P)=X(\tilde P)$.
The dual cone $\sigma =C^{\vee}\subset N_{\mathbb R}$ to the recession cone
may not be strongly convex, while $\sigma$ is the union of the maximal cones in $\Sigma _{\tilde P}$.
Let $W = \mathbb R\mbox{-span}(\sigma \cap (-\sigma))$ and define an affine toric variety
$U(\tilde P)=X_{N/W\cap N}(\sigma ')$, where $\sigma '$ is the image of $\sigma $ in $N_{\mathbb R}/W$.
Then, the toric variety $X(\tilde P)$ is equipped with a projective toric morphism
\[
\phi _P : X(\tilde P) \to U(\tilde P).
\]
A projective toric morphism over an affine toric variety can always be realized
as $\phi _P$ above (\emph{op. cit.}, Theorem 7.2.4 and Proposition 7.2.3).
A projective birational toric modification is a special case in which
the cone $\sigma$ is strictly convex and $U(\tilde P)=X(\sigma)$.

\subsection{Toric blowing-up}
Let $\sigma \subset N_{\mathbb R}$ be a rational polyhedral cone
and $m_1,\dots, m_r\in M$ a set of generators of the semigroup $\sigma ^{\vee}\cap M$.
Then, the affine toric variety $X(\sigma)$ is the closure of the image of a map
\[
T\to \mathbb A^r;\quad t\mapsto (\chi ^{m_1}(t),\dots, \chi ^{m_r}(t))
\]
where $\chi ^m$ is a monomial with the exponent $m$.
Let us take another set of elements $m'_1,\dots, m'_s\in M$ and consider
a polytope $P$ that is the convex hull of  $\{m'_0=0, m'_1,\dots, m'_s\}$ in
$M_{\mathbb R}$ and a polyhedron $\tilde P$ given as the Minkowski sum $P+\sigma ^{\vee}$.
Let us assume further the following.

\begin{assumption}\label{very ampleness}
For every vertex $v\in \tilde P$,
the set
\[
\{m_1-v,\dots, m_r-v,-v,m'_1-v,\dots, m'_s-v\}
\]
generates the semigroup
$\sigma _v^{\vee}\cap M$ where $\sigma _v$ is the normal cone to $\tilde P$ at $v$.
\end{assumption}

\noindent
Note that this assumption
immediately implies that the polyhedron $\tilde P$ is very ample (\emph{op. cit.}, Definition 7.1.8).
In this situation, the toric variety $X(\tilde P)$ can be realized as the closure of the image of
a monomial map
\[
T \to  \mathbb A^r\times \mathbb P^s;\quad
t\mapsto \left( (\chi ^{m_1}(t),\dots ,\chi ^{m_r}(t)),\,
[1:\chi ^{m'_1}(t):\dots :\chi ^{m'_s}(t)] \right).
\]
The canonical morphism $\phi _P:X(\tilde P)\to X(\sigma)$, which we call
a \emph{toric blowing-up}, is nothing but the projection to the first factor $\mathbb A^r$.

\subsection{Fractional linearization and torus quotient}\label{frac lin}

Now we consider an action of a sub-torus $G\subset T$ on
a toric blowing-up $X(\tilde P)$ and
discuss GIT quotients of $X(\tilde P)$ by $G$. The following argument is
a slight generalization of \cite{KSZ}, \S 3.

The sub-torus $G\subset T$ acts in a trivial way on $X(\tilde P)$ and $X(\sigma)$
such that $\phi _{\tilde P}$ is $G$-equivariant.
Let $L$ be a line bundle on $X(\tilde P)$
that is the pull back of $\mathcal O_{\mathbb P^s}(1)$.
We call a linearization of $L^{\otimes k}$ a \emph{fractional linearization} of $L$

The cone $C(\tilde P)$ associated to $\tilde P$ is a cone in $\tilde M_{\mathbb R}
=M_{\mathbb R}\times \mathbb R$ such that
\[
C(\tilde P)\cap H_t = t\tilde P = tP + \sigma ^{\vee}
\]
where $H_t=\{(m,t)\in M_{\mathbb R}\times \mathbb R\; |\; m\in M_{\mathbb R}\}$,
the hyperplane of `height $t$'. The cone $C(\tilde P)$,
in turn, determines a graded ring
\[
S(\tilde P)=\mathbb C[C(\tilde P)\cap (M\times \mathbb Z)],
\]
and we know that $X(\tilde P)\cong \Proj \; S(\tilde P)$ (\emph{op. cit.}, Theorem 7.1.13).
Note that $S(\tilde P)_0=\mathbb C[\sigma ^{\vee}\cap M]$ is the coordinate
ring of the affine toric variety $X(\sigma)$ and
$S(\tilde P)_k = H^0(X(\tilde P),L^{\otimes k})
= \bigoplus _{m\in k\tilde P} \mathbb C \chi ^m$.

\begin{proposition}\label{torus quotient via polyhedron}
  We keep the notation above. Let $M_G$ be the character lattice of $G$ and
  $\alpha :M\to M_G$ the canonical projection corresponding to the embedding $G\subset T$.
  \begin{enumerate}[\rm (1)]
    \item The set of fractional linearizations of $L$ is naturally identified with
    $M_G\otimes \mathbb Q$.
    \item Assume that $b\in M_G\otimes \mathbb Q$ is a fractional $G$-linearlization of $L$.
    Then, the GIT quotient $X(\tilde P)^{ss}(L,b)\git G$ is a toric variety given by
    a polyhedron
    \[\tilde P_b = \tilde P \cap (\alpha \otimes \mathbb R) ^{-1}(-b),
    \]
    which is naturally identified with a (fractional) lattice polyhedron on
    a sublattice $\Ker(\alpha)\subset M$.
    \end{enumerate}
\end{proposition}

\begin{proof}
(1) After passing to sufficiently high Veronese embedding, namely passing to $L^{\otimes m}$ instead
of $L$, if necessary, we may assume that $S(\tilde P)$ is genarated
by $S(\tilde P)_1=\mathbb C[C(\tilde P)\cap (M\times \{ 1\})]$, \emph{i.e.}, we assume that
$\tilde P$ is a \emph{normal} polyhedron (\emph{op. cit.}, Definition 7.1.8).
Then, to give a $G$-linearization of $L$ is
the same as to give a dual $G$-action on the $S(\tilde P)_0$-module
\begin{equation}\label{module linearization}
S(\tilde P)_1\to S(\tilde P)_1\otimes \mathbb C[G]
\end{equation}
that is compatible with the dual $G$-action on $S(\tilde P)_0$ (cf. \cite{Muk}, Definition 6.23).
The $G$-action on $S(\tilde P)_0=\mathbb C[\sigma ^{\vee}\cap M]$ is determined by
the canonical projection $\alpha: M\to M_G$. The map \eqref{module linearization} is determined by
a map
\begin{equation}\label{linearization on monomials}
l:\tilde P\cap M\to M_G
\end{equation}
satisfying $l(m'+m)=\alpha(m')+l(m)$. This immediately implies that the map $l$ is
(a restriction of) an affine map $l:M\to M_G$ such that $l(m)=\alpha (m)+l(0)$ for all
$m \in M$. Therefore, a fractional linearization of $L$ is
in one to one correspondence with $b = l(0)\in M_G\otimes \mathbb Q$.

\noindent
(2)
A (integral) linearization $b\in M_G$ determines
a diagonal action of $G$ on
$S(\tilde P)_k = \bigoplus _{m\in k\tilde P} \mathbb C \chi ^m$,
thus it determines the ring of invariants $S(\tilde P)^{(G,b)}$
with respect to this action.
A monomial function $\chi ^m$ is $G$-invariant if and only if $l_b(m)=0$
for $l_b=l=\alpha + b$ as in \eqref{linearization on monomials} corresponding
to $b$.
The GIT quotient $X^{ss}(L,b)\git G$ is defined to be the $\Proj$ of
the graded ring $S(\tilde P)^{(G,b)}$.
If we define an affine plane $\tilde M_{\mathbb R,b}\subset \tilde M_{\mathbb R}$ by
\[
\tilde M_{\mathbb R,b}=\{(m,t)\in M_{\mathbb R}\times \mathbb R \; | \;
(l_b\otimes \mathbb R) (m)=0\},
\]
the invariant ring is given by
\[
S(\tilde P)^{(G,b)} =
\mathbb C[C(\tilde P)\cap \tilde M_{\mathbb R,b}\cap (M\times \mathbb Z)],
\]
which is nothing but $S(\tilde P_b)$.
Therefore, the GIT quotient $X(\tilde P)^{ss}(L,b)\git G$ is
the toric variety associated with $\tilde P_b$.
The case of fractional linearization $b\in M_G\otimes \mathbb Q$,
we just pass to a sufficiently high truncation of
the graded ring $S(\tilde P)$.
\end{proof}

We note that if we set $\bar \sigma ^{\vee} = \sigma^{\vee} \cap \Ker (\alpha) \otimes \mathbb R$
and $P_b = P \cap (\alpha \otimes \mathbb R)^{-1}(-b)$, we have $\tilde P_b = P_b+\bar \sigma^{\vee}$,
that is, the GIT quotient $X(\tilde P)^{ss}(L,b)\git G$ is
the toric blow-up of the affine quotient $X(\sigma)\git G$
determined by the polytope $P_b$.

\section{Toric description of a family of expanded degeneration}

\subsection{Family of expanded degenerations $X[n]$}
Now we study the local model of expanded degeneration using toric geometry.
For details, we refer \cite{GHH}. We also follow the notation in \emph{op. cit.}

Let $X=\mathbb A^2$ and $C=\mathbb A^1$ with coordinates $(x,y)$ and $t$, respectively,
and consider the morphism $X\to C=\mathbb A^1$ defined by $t=xy$. The base change
$X\times _{\mathbb A^1} \mathbb A^{n+1}$ by
\[
\mathbb A^{n+1}\to \mathbb A^1;\quad (t_1,\dots, t_{n+1})\mapsto t_1\dots t_{n+1}
\]
is an affine variety defined by $xy-t_1\dots t_{n+1}=0$ in $\mathbb A^{n+3}$.
The $n$-th family of expanded degeneration $X[n]\to \mathbb A^{n+1}$ is
a successive blowing-up $X[n]$ of $X\times _{\mathbb A^1} \mathbb A^{n+1}$
by the strict transform of a subvariety defined by the ideal $(t_i,x)$ for
$i=1,\dots, n$ in this order, equipped with the natural projection to $\mathbb A^{n+1}$.

One can easily see that $X[n]$ is a toric variety by the construction.
Let us describe $X[n]$ via a polyhedron.
It is easy to see that $X[n]$ is the closure of
the image of a map
\[
\Phi' [n]:T[n]=(\mathbb C^*)^{n+2}\to (\mathbb A^2\times \mathbb A^{n+1})\times (\mathbb P^1)^{n}
\]
defined by
\begin{multline*}
(s, t_1,\dots, t_{n+1})\mapsto \\
(\, \left(\frac{t_1\dots t_{n+1}}{s},s,\; t_1,\dots, t_{n+1}\right),\;
\left[1:\frac{t_2\dots t_{n+1}}{s}\right],
\left[1:\frac{t_3\dots t_{n+1}}{s}\right],\dots,
\left[1:\frac{t_{n+1}}{s}\right]\, ).
\end{multline*}
Here we note that we have the relations
\[
x = \frac{t_1\dots t_{n+1}}{s}\mbox{\quad and\quad }y = s.
\]
Composing with the Segre embedding $(\mathbb P^1)^n\to \mathbb P^{2^n-1}$,
we get a monomial map
\[
\Phi [n]:T[n] \to \mathbb A^{n+3}\times \mathbb P^{2^n-1}.
\]
By the description in \S 1.1, we have a polyhedron $\tilde P[n]$ on
the character lattice $M[n]=\mathbb Z^{n+3}$ of $T[n]$
such that
\[
  \phi _{\tilde P[n]}:X[n] = X(\tilde P[n])\to  X\times _{\mathbb A^1}\mathbb A^{n+1} = X(\sigma[n])
\]
is the composite of blowing-ups described above.
Here $\sigma [n]$ is the dual cone of the recession cone of the polyhedron $\tilde P[n]$, namely
we have a Minkowski sum decomposition
\[
\tilde P[n] = P[n] + \sigma [n]^{\vee}
\]
with $P[n]$ a lattice polytope on $M[n]$.

The cone $\sigma [n]^{\vee}$ is easy to describe:
taking a basis of $M[n]$ corresponding to
the coordinate $(s,t_1,\dots, t_{n+1})$, $\sigma [n]^{\vee}\subset M[n]_{\mathbb R}$ is
the cone generated by the column vectors of the matrix
\[
\sigma [n]^{\vee} =
\begin{pmatrix}
-1 & 1 & 0 & \cdots & 0 \\
1 & 0 & 1 & \cdots & 0 \\
\vdots & \vdots & \vdots &\ddots & \vdots \\
1 & 0 & 0 & \cdots & 1
\end{pmatrix},
\]
corresponding to the monomials $x,y,t_1,\dots, t_{n+1}$ in this order.

Let $\square_{n}$ be a hypercube in $\mathbb R^{n}$ whose $2^n$ vertices are
the vectors whose entries are $0$ or $1$. Then we define $P[n]$ to be the image of
$\square_n$ under the linear map $\mathbb R^n \to \mathbb R^{n+2}=M_{\mathbb R}$
defined by the left multiplication of a matrix
\[
\begin{pmatrix}
  -1 & -1 & \cdots & -1 & -1 \\
  0 & 0 & \cdots & 0 & 0 \\
  1 & 0 & \cdots & 0 & 0 \\
  1 & 1 & \ddots & 0 & 0 \\
  \vdots & \vdots & \ddots & \vdots & \vdots \\
  1 & 1 & \cdots & 1 & 0 \\
  1 & 1 & \cdots & 1 & 1
\end{pmatrix}.
\]
It is straightforward to see that $P[n]$ is a lattice polytope whose vertices
are generated by the vectors corresponding to the monomials that appear as the entries
of the monomial map $\pr_2\circ \Phi[n]: T\to \mathbb P^{2^n-1}$.
It is also easy to check that the polyhedron $\tilde P[n]$ is very ample,
thus it satisfies Assumption \ref{very ampleness}.

\subsection{Self-product $W[n]$}\label{def W[n]}
For later use, we calculate the polyhedron $\tilde P_W[n]$
corresponding to the $n$-fold self-product of $X[n]$ over the base $\mathbb A^{n+1}$,
\[
W[n] = (X[n]/\mathbb A^{n+1})^n = X[n] \times _{\mathbb A^{n+1}} \dots \times _{\mathbb A^{n+1}} X[n].
\]
It is easy to see that $W[n]$ is the closure of the image of a map
\[
\Phi' _W [n]:T_W[n]:=(\mathbb C^*)^{2n+1}\to
((\mathbb A^2)^n\times \mathbb A^{n+1})\times ((\mathbb P^1)^n)^n)
\]
defined by
\begin{multline}\label{coordinte of W[n]}
(s_1,\dots, s_n,t_1,\dots, t_{n+1})\mapsto \\
(\, (\, \left(\frac{t_1\dots t_{n+1}}{s_1},s_1;\,\dots \,;
\frac{t_1\dots t_{n+1}}{s_n},s_n\right),\, (t_1,\dots ,t_{n+1})\, ), \qquad \qquad \\
(\left[1:\frac{t_2\dots t_{n+1}}{s_1}\right],
\left[1:\frac{t_3\dots t_{n+1}}{s_1}\right],\dots,
\left[1:\frac{t_{n+1}}{s_1}\right]),\; \cdots\;  , \\
(\left[1:\frac{t_2\dots t_{n+1}}{s_n}\right],
\left[1:\frac{t_3\dots t_{n+1}}{s_n}\right],\dots,
\left[1:\frac{t_{n+1}}{s_n}\right])
).
\end{multline}
Let $M_W[n]$ be the character lattice of $T_W[n]$.
The recession cone $\sigma _W[n]^{\vee}$ of $\tilde P_W[n]$
is a rational polyhedral cone on $M_W[n]$
generated by the column vectors of a $(2n+1,3n+1)$ matrix
\begin{equation}\label{sigma W dual}
\sigma _W[n]^{\vee} =
\left(
\begin{array}{ccc;{2pt/2pt}c;{2pt/2pt}c}
  &-I_n& &I_n & O\\ \hdashline[2pt/2pt]
  1 & \cdots & 1 & & \\
  \vdots & & \vdots & O & I_{n+1} \\
  1 & \cdots & 1 & &
\end{array}
\right),
\end{equation}
while the polytopal part $P_W[n]$ is the image of the hypercube
$\square _{n^2}\subset \mathbb R^{n^2}$ by
a linear map defined by the matrix
\begin{equation}\label{proj from hypercube}
L[n] =
\left(
\begin{array}{ccc;{2pt/2pt}ccc;{2pt/2pt}c;{2pt/2pt}ccc}
 &&&&&&&&\\
 & -I_n & & & -I_n & & \cdots & & -I_n & \\
 &&&&&&&&\\\hdashline[2pt/2pt]
 0 & \cdots & 0 & 0 & \cdots & 0 & \cdots & 0 & \cdots & 0 \\
 1 & \cdots & 1 & 0 & \cdots & 0 & \cdots & 0 & \cdots & 0 \\
 1 & \cdots & 1 & 1 & \cdots & 1 & \cdots & 0 & \cdots & 0 \\
 & \vdots & & & \vdots & & \ddots & &\vdots & \\
 1 & \cdots & 1 & 1 & \cdots & 1 & \cdots & 1 & \cdots & 1
\end{array}
\right)
\end{equation}
of size $(2n+1,n^2)$.
As $\tilde P[n]$ satisfies Assumption \ref{very ampleness},
$\tilde P_W[n]$ also satisfies the assumption.

\section{relative symmetric product of an expanded degeneration and its quotient}

\subsection{Small resolution $\tilde Z^{(n)\prime}$ of $(X/C)^n$}
Next we review the construction of a small crepant resolution $\tilde Z^{(n)\prime}$ of
the relative $n$-fold self-product $(X/C)^n$ of the family $X\to C$ in \cite{N}.
For details, we refer \emph{op. cit.}, \S\S 1 and 2.

Let $\tilde X^{(n)\prime}=(X/C)^n=X\times _C \dots \times _C X$
be the $n$-fold self-product of $X$ over $C$. It is an affine toric variety defined by
the equations
\[
z_{11}z_{12}=z_{21}z_{22}=\dots =z_{n1}z_{n2}
\]
in $\mathbb A^{2n}$ with coordinates $(z_{11},z_{12},\dots, z_{n1},z_{n2})$.
The symmetric group $\mathfrak S_n$ acts on $X^{(n)\prime}$ by the permitation of
the first index $i$ of $z_{ij}$.

Let $\overline N=\mathbb Z^{n-1}$ and $e_i\in \overline N$ a vector whose
$i$-th entry is one and all the other entries are 0.
For a nonempty subset $I\subset \{1,2,\dots, n\}$, we define
$e_I= \sum _{i\in I} e_i$
and call such vectors \emph{primitive positive weight vectors}.
We define primitive negative weight vectors as the negation of
positive vectors. The positive vectors and negative vectors span a full dimensional
smooth fan $\bar \Delta ^{n}$ in $\overline N_{\mathbb R}$, which is isomorphic to
the Coxeter complex of $A_{n-1}$-root system. The simple reflections acts on $\overline N$ by
\[
(k\ \ k+1)=
\left(\begin{array}{c;{2pt/2pt}cc;{2pt/2pt}c}
I_{k-1} & & & \\ \hdashline[2pt/2pt]
& 0 & 1 & \\
& 1 & 0 & \\ \hdashline[2pt/2pt]
& & & I_{n-k-2}
\end{array}\right)
\mbox{\; and \;}
(n-1\ \ n)=
\begin{pmatrix}
1 & 0 & \cdots & 0 & -1 \\
0 & 1 & \cdots & 0 & -1 \\
\vdots & \vdots & \ddots & \vdots & -1\\
0 & 0 & \cdots & 1 & -1 \\
0 & 0 & \cdots & 0 & -1
\end{pmatrix}.
\]
One can easily check that the cone $\bar \delta ^{(n)}$ generated by the column vectors of
\[
\bar \delta ^{(n)}=
\begin{pmatrix}
  1 & 1 & \cdots & 1 \\
  0 & 1 & \cdots & 1 \\
  \vdots &\vdots  & \ddots & \vdots \\
  0 & 0 & \cdots & 1
\end{pmatrix}
\]
and its $\mathfrak S_n$-translates are exactly the maximal cones of the fan $\bar \Delta ^{(n)}$.
We have the correspoinding projective toric variety $X(\bar \Delta ^{(n)})$.
We denote the variety by $X(A_{n-1})$ for simplicity.

Let $N = \mathbb Z \oplus \overline N\oplus \mathbb Z
=\mathbb Z^{n+1}$ and consider an $\mathfrak S_n$-action
defined by
\begin{equation}\label{action on N}
\begin{aligned}
(k \ \ k+1) &=
\left(
\begin{array}{c;{2pt/2pt}cc;{2pt/2pt}c}
I_k & & & \\ \hdashline[2pt/2pt]
& 0 & 1 & \\
& 1 & 0 & \\ \hdashline[2pt/2pt]
& & & I_{n-k-1}
\end{array}
\right)
\mbox{\; for $k = 1,\dots, n-2$ and \;} \\
(n-1 \ \ n) &=
\left(
\begin{array}{c;{2pt/2pt}ccccc;{2pt/2pt}c}
1 & 0 & 0 & \cdots & 0 & -1 & 0 \\ \hdashline[2pt/2pt]
0 & 1 & 0 & \cdots & 0 & -1 & 0 \\
0 & 0 & 1 & \cdots & 0 & -1 & 0 \\
\vdots & \vdots & \vdots & \ddots & \vdots & \vdots & \vdots \\
0 & 0 & 0 & \cdots & 1 & -1  & 0  \\
0 & 0 & 0 & \cdots & 0 & -1 & 0\\ \hdashline[2pt/2pt]
0 & 0 & 0 & \cdots & 0 & 1 & 1
\end{array}
\right).
\end{aligned}
\end{equation}
Then, the projection $N\to \overline N$ is
$\mathfrak S_n$-equivariant. Let $\delta ^{(n)}$ be the maximal cone in $N_{\mathbb R}$
spanned by the column vectors of
\[
\delta ^{(n)}=
\left(
\begin{array}{c;{2pt/2pt}cccc;{2pt/2pt}c}
  1 & 1 & 1 &\cdots & 1 & 0 \\ \hdashline[2pt/2pt]
  0 & 1 & 1 &\cdots & 1 & 0 \\
  0 & 0 & 1 & \cdots & 1 & 0 \\
  \vdots & \vdots & \vdots & \ddots & \vdots & \vdots \\
  0 & 0 & 0 & \cdots & 1 & 0 \\ \hdashline[2pt/2pt]
  0 & 0 & 0 & \cdots & 0 & 1
\end{array}
\right)
\]
and $\Delta ^{(n)}$ be the fan consiting of faces of maximal cones
$s\delta ^{n}\; (s\in \mathfrak S_n)$. Then, it is easy to see that
the toric variety $\tilde Z^{(n)\prime} = X(\Delta ^{(n)})$
is the total space of $\mathbb C^2$-bundle
$\mathcal O_{X(A_{n-1})}(-D_{pos})
\oplus \mathcal O_{X(A_{n-1})}(-D_{neg})$ over $X(A_{n-1})$, where
$D_{pos}$ is the sum of torus invariant divisors corresponding to positive vectors
and $D_{neg}$ is defined similarly;
\[
D_{pos}= \sum _I D_{e_I},\quad
D_{neg}= \sum _I D_{-e_I}.
\]
Let $\sigma ^{(n)}$ be the union of all the maximal cones in $\Delta ^{(n)}$.
It is easy to see that no ray in $\Delta ^{(n)}$ is in the relative interior of
$\sigma ^{(n)}$, and that $\sigma ^{(n)}$ is generated by the column vectors of
\begin{equation}\label{sigma (n)}
\sigma ^{(n)}=
\left(
\begin{array}{c;{2pt/2pt}cccc;{2pt/2pt}cccc;{2pt/2pt}c}
1 & 1 & \cdots & \cdots & 1 & 0 & \cdots &\cdots & 0 &  0 \\ \hdashline[2pt/2pt]
0 & &&& & &&& & 0 \\
\vdots & \multicolumn{4}{|c|}{\shortstack{positive primitive\\ weight vectors}}
  & \multicolumn{4}{|c|}{\shortstack{negative primitive\\ weight vectors}} & \vdots \\
0 & &&& & &&& & 0 \\ \hdashline[2pt/2pt]
0 & 0 & \cdots & \cdots & 0 & 1 & \cdots & \cdots & 1 & 1
\end{array}
\right).
\end{equation}
This implies that the associated projective birational toric morphism
$X(\Delta ^{(n)})\to X(\sigma ^{(n)})$ is small \emph{i.e.}, its exceptional set
has no divisorial component.

\begin{proposition}[\cite{N}, Proposition 1.4, Proposition 2.5]\label{small resolu}
 The affine toric variety $X(\sigma ^{(n)})$ is the relative n-fold product
 $\tilde X^{(n)\prime}$ of $X$ over $C$. Therefore, $\tilde Z^{(n)\prime}=X(\Delta ^{(n)})$
 is an $\mathfrak S_n$-equivariant small projective resolution
 of $\tilde X^{(n)\prime}$.
\end{proposition}

We can also describe the toric variety $X(\Delta ^{(n)})$ in terms of polyhedron.
It is well-known that the Coxeter complex $\bar \Delta ^{(n)}$ of $A_{n-1}$-root system is
a normal fan to the $n$-th \emph{permutahedraon} $P^{(n)}$. One of a realization
of $P^{(n)}$ is as follows; define $P^{(n)}$ as the convex hull of the vertex set
\[
\Big\{ v_s=
  \begin{pmatrix} s(1) \\ s(2) \\ \vdots \\ s(n-1)\end{pmatrix}
    - \begin{pmatrix} 1 \\ 2 \\ \vdots \\ n-1 \end{pmatrix}
  \; \Big| \;  s\in \mathfrak S_n
\Big\} \subset \mathbb R^{n-1}.
\]
This is clearly a lattice polytope on $\overline M=\mathbb Z^{n-1}$, the dual of $\overline N$.
The normal Fan to $P^{(n)}$ agrees with our $\bar \Delta ^{(n)}$. Actually, the vertices
adjacent to $v_e = (0,\dots, 0)^T$ is given by $v_s$ for all simple transpositions
$s=(1\; 2), (2\; 3), \dots, (n-1\; n)$, namely the column vectors of
\begin{equation}\label{edge cone}
B_e=
\begin{pmatrix}
1 & 0 & \cdots & \cdots & 0 & 0\\
-1 & 1 & \cdots & \cdots & 0 & 0\\
0 & -1 & \ddots &  & 0 & 0 \\
\vdots & \vdots & \ddots &\ddots & \vdots & \vdots \\
0 & 0 &  &\ddots & 1 & 0\\
0 & 0 & \cdots &\cdots & -1 & 1
\end{pmatrix},
\end{equation}
The dual cone to the cone $B_e$ spanned by the column vectors of the matrix
is the normal cone to $P^{(n)}$ at $v_e$, which
agrees with the positive Weyl chamber $\bar \delta ^{(n)}$.
Moreover, one can easily check that the normal cone at $v_{s^{-1}}\; (s\in \mathfrak S_n)$
is $s\bar \delta ^{(n)}$.

Now we consider $M=\mathbb Z\oplus \overline M\oplus \mathbb Z$, the dual of
$N=\mathbb Z\oplus \overline N\oplus \mathbb Z$ and let $\iota P^{(n)}$
be the image of $P^{(n)}$
under the natural injection $\iota: \overline M\to M$. We define a polyhedron
\[
\tilde P^{(n)} = \sigma ^{(n)\vee} + \iota P^{(n)}.
\]

\begin{proposition}\label{polytope of Znp}
The toric variety $X(\tilde P^{(n)})$ associated to the polyhedron $\tilde P^{(n)}$
is isomorphic to $X(\Delta ^{(n)})$.
\end{proposition}

\begin{proof}
Note that the set of vertices of $\tilde P^{(n)}$ is the same as the set of vertices
of $\iota P^{(n)}$,
\[
\Big\{
\tilde v_s =
\left(
\begin{array}{c}
0 \\ \hdashline[2pt/2pt]
v_s \\ \hdashline[2pt/2pt]
0
\end{array}
\right) \in \mathbb R \oplus \overline M_{\mathbb R}\oplus \mathbb R\; \Big| \;
s\in \mathfrak S_n)
\Big\}.
\]
The cone $\sigma ^{(n)\vee}$ is generated by the column vectors of
\[
\sigma ^{(n)\vee}=
\begin{pmatrix}
1 & 1 & 1 & \cdots & 1 & 0 & 0 & 0 & \cdots & 0 \\
0 & -1 & 0 & \cdots & 0 & 0 & 1 & 0 & \cdots & 0 \\
0 & 0 & -1 & \cdots & 0 & 0 & 0 & 1 & \cdots & 0 \\
\vdots & \vdots & \vdots & \ddots & \vdots & \vdots & \vdots & \vdots & \ddots & \vdots \\
0 & 0 & 0 & \dots & -1 & 0 & 0 & 0 & \cdots & 1 \\
0 & 0 & 0 & \cdots & 0 & 1 & 1 & 1 & \cdots & 1
\end{pmatrix}
\]
by \emph{op. cit.}, \S 1.6 (note that we are working on a basis modified by $Q$
in \emph{op. cit.}, Proof of Proposition 2.5).
One can easily see that a cone $\sigma ^{(n)\vee} + \iota B_e$
is generated by the column vectors of
\begin{equation}\label{delta-n-dual}
\begin{pmatrix}
1 & 0 & \cdots & \cdots & 0 & 0 & 0 \\
-1 & 1 & \cdots & \cdots & 0 & 0 & 0 \\
0 & -1 & \ddots & & 0 & 0 & 0 \\
\vdots & \vdots & \ddots & \ddots & \vdots & \vdots & \vdots \\
0 & 0 & & \ddots & 1 & 0 & 0 \\
0 & 0 & \cdots & \cdots & -1 & 1 & 0 \\
0 & 0 & \cdots & \cdots & 0 & 0 & 1
\end{pmatrix}.
\end{equation}
Therefore, the normal cone to $\tilde P^{(n)}$ at $\tilde v_e$, which is the
dual cone to $\sigma ^{(n)\vee} + \iota B_e$, agrees with $\delta ^{(n)}$.
At the vertex $\tilde v_{s^{-1}}\; (s\in \mathfrak S_n)$, the normal cone is the dual
cone to $\sigma ^{(n)\vee}+\iota (s^{-1}B_e)$ as $\sigma ^{(n)\vee}$ is invariant
under the action of $\mathfrak S_n$. However, as we know that the dual cone
of $s^{-1}B_e$ is nothing but $s\bar \delta ^{(n)}$, the normal cone to $\tilde P^{(n)}$
at $\tilde v_{s^{-1}}$ must be the same as $s\delta ^{(n)}$. This implies that
the normal fan of $\tilde P^{(n)}$ is exactly the fan $\Delta ^{(n)}$.
\end{proof}

\subsection{GHH linearization}\label{GHH linearization}
Let us go back to the self-product $W[n]$ of the expanded degeneration $X[n]$.
We keep the notation in \S \ref{def W[n]}.
Gulbrandsen, Halle, and Hulek introduced in \cite{GHH} a specific fractional linearization
on the expanded degeneration $X[n]$ with respect to the embedding
\[
\Phi [n]: X[n]\overset{\Phi '[n]}{\lto}
(\mathbb A^2\times \mathbb A^{n+1})\times (\mathbb P^1)^n
\overset{\mbox{\scriptsize Segre}}{\lto}
(\mathbb A^2\times \mathbb A^{n+1})\times \mathbb P^{2^n-1}.
\]
Here we describe the fractional linearization in the framework of \S \ref{frac lin}.

Let us consider the torus $(\mathbb C^*)^{n+1}$ of the base space of the $n$-th expanded degeneration
$X[n]\to \mathbb A^{n+1}$ with coordinate $(t_1,\dots, t_{n+1})$. We define
\[
G[n]:=\{(t_1,\dots, t_n,t_{n+1})\in (\mathbb C^*)^{n+1}\; |\; t_1\dots t_n t_{n+1}=1\}.
\]
We note that we can naturally regard $G[n]$
as a sub-torus of $T[n]$ or $T_W[n]$ by our consistent choice of coordinate.
We also note that $G[n]$ has a natural action of $G[n]$ on $(\mathbb P^1)^n$ through
the map $\Phi '[n]$.
Following \cite{GHH}, we introduce another coordinate $(\tau _1,\dots ,\tau _n)$ of $G[n]$ by
\[
\tau _i =  \prod _{j=1}^i t_j\quad (i=1,\dots, n).
\]
In other words, we have $t_1 = \tau _1$ and $t_i = \tau _i / \tau _{i-1}$ for $i=2,\dots, n+1$.
Now we let $G[n]$ act on $(\mathbb A^2)^n$ by
\[
(\dots, (u_i,v_i), \dots )\mapsto (\dots, (\tau _i^{\frac{i}{n+1}}u_i,\tau _i^{\frac{i}{n+1}-1}v_i),\dots).
\]
Putting
\[
[u_i:v_i] = \left[1:\frac{t_{i+1}\dots t_{n+1}}{s} \right],
\]
we see that this is a lifting of the $G[n]$ action on $(\mathbb P^1)^n$ induced by $\Phi '[n]$.
This determines a linearlization on $L^{\otimes n+1}$ where $L$ is a pull-back to $X[n]$ of
$\mathcal O_{\mathbb P^{2^n-1}}(1)$ under $\Phi [n]$, hence we get a fractional linearlization on $L$, which
we call \emph{GHH fractional linearlization}. As we saw in \S \ref{frac lin}, we have a corresponding
affine map
\[
l _{GHH} : M[n]_{\mathbb Q} \to M_G[n]_{\mathbb Q},
\]
where we will always take a dual basis on $M_G[n]$
to the coordinate $(\tau _1,\dots, \tau _n)$.
The origin of $M[n]$ is a vertex of the polytope $P[n]$ that correspoinds to
the monomial $u_1\dots u_n$. As $\tau _i$ acts on the monomial via a character
$\tau _i\mapsto \tau _i^{\frac{i}{n+1}}$, we know that
\[
b _{GHH} = l_{GHH}(0) =
\begin{pmatrix}
\frac{1}{n+1} \\ \vdots \\ \frac{n}{n+1}
\end{pmatrix}.
\]

The GHH fractional linearization induces a fractional linearization of the self-product
$W[n]$ with respect to the embedding
\[
\Phi _W [n]:W[n]\overset{\scriptsize \Phi '_W [n]}{\lto}
((\mathbb A^2)^n\times \mathbb A^{n+1})\times ((\mathbb P^1)^n)^n)
\overset{\mbox{\scriptsize Segre}}{\lto}
((\mathbb A^2)^n\times \mathbb A^{n+1})\times \mathbb P^{2^{n^2}-1} .
\]
The origin of $M_W[n]$ is a vertex of the polytope $P_W[n]$ corresponding to
a monomial $(u_1 \dots u_n)^n$, and therefore the induced fractional linearization is given by
\[
b^{W[n]}_{GHH} =
\begin{pmatrix}
  \frac{n}{n+1} \\ \frac{2n}{n+1} \\ \vdots \\ \frac{n^2}{n+1}
\end{pmatrix}.
\]
For later use, we note that the linear part $\alpha _W[n]:M_W[n]\to M_G[n]$
of the GHH linearlization
\[
l^{W[n]}_{GHH}=\alpha _W[n]+b^{W[n]}_{GHH}:
M_W[n]_{\mathbb Q}\to M_G[n]_{\mathbb Q}
\]
is given by the $(n,2n+1)$ matrix
\[
\alpha _W[n] =
\left(
\begin{array}{c;{2pt/2pt}cccccc}
  & 1 & -1 & 0 & \cdots & 0 & 0 \\
  & 0 & 1 & -1 & \cdots & 0 & 0 \\
O_n\;\;  & \vdots  & \vdots   & \vdots & \ddots & \vdots & \vdots \\
  & 0 & 0 & 0 & \cdots & -1 & 0 \\
  & 0 & 0 & 0 & \cdots & 1 & -1
\end{array}
\right),
\]
where $O_n$ is the zero matrix of size $(n,n)$.

\subsection{GIT quotient of $W[n]$}

Now we are prepared to prove the following

\begin{theorem}\label{quotiont of W[n]}
Notation as above.
The GIT quotient $W[n]^{ss} \git G[n]$ with respect to
the GHH fractional linearlization is isomorphic to $\tilde Z^{(n)\prime}$.
\end{theorem}

According to Proposition \ref{torus quotient via polyhedron},
the polyhedron
\[
\tilde P_b[n] := \tilde P_W[n]\cap (\alpha _W[n]\otimes \mathbb R)^{-1}(-b)
\]
with $b=b^{W[n]}_{GHH}$ determines the quotient $W[n]^{ss}\git G[n]$.
First we calculate the recession cone $\sigma '[n]^{\vee}$
of $\tilde P_b[n]$. We have a short exact sequence of lattices
\[
0 \lto M'[n] \lto M_W[n] \overset{\alpha _W[n]}{\lto} M_G[n]\lto 0.
\]
By dualizing the sequence, we get a surjective map
$\pi:N_W[n]\to N'[n]\cong \mathbb Z^{n+1}$. Then by \cite{H}, Lemma 10.1,
the dual to the recession cone $\sigma '[n]$ is just the image
$\pi (\sigma _W[n])$. By making an appropriate choice of basis for $N'[n]$,
$\pi$ is given by the matrix
\begin{equation}\label{matrix pi}
\pi =
\left(
\begin{array}{ccccc;{2pt/2pt}ccccc}
0 & 0 & \cdots & 0 & -1 & 1 & 1 & \cdots & 1 & 1 \\ \hdashline[2pt/2pt]
1 & 0 & \cdots & 0 & -1 & &\\
0 & 1 & \cdots & 0 & -1 & & \\
\vdots & \vdots & \ddots & \vdots & \vdots & & & O_{n,n+1}\\
0 & 0 & \cdots & 1 & -1 & & \\
0 & 0 & \cdots & 0 & 1 & & \\
\end{array}
\right).
\end{equation}

\begin{lemma}\label{conical part}
{\rm (1)} The dual cone $\sigma _W[n]$ of the recession cone of $\tilde P_W[n]$ is
a rational polyhedral cone genereted by the vectors
\[
v =
\left(
\begin{array}{c}
  a_1 \\ \vdots \\ a_n \\ \hdashline[2pt/2pt]
  b_1 \\ \vdots \\ b_{n+1}
\end{array}
\right)
\]
such that
(i) $a_i$ is either $0$ or $1$ for $i=1,\dots, n$, and
(ii) exactly one among $b_j$'s is $1$ and others are all $0$.

\noindent {\rm (2)} The image $\sigma '[n] = \pi (\sigma _W[n])$
coincides with $\sigma ^{(n)}$.

\end{lemma}

\begin{proof}
  We have a list \eqref{sigma W dual} of generators for $\sigma _W[n]^{\vee}$.
  The two blocks on the right implies that if $v\in \sigma _W[n]$, all $a_i$ and $b_j$ are
  non-negative. The condition from leftmost block is
  \[
  a_i \leqslant \sum _{j=1}^{n+1} b_j \quad (i=1,\dots, n).
  \]
  Thus, $\sigma _W[n]$ is a family of hypercubes in $(a_1,\dots, a_n)$
  of size $\sum b_j$ over
  the positive orthant in $(b_1,\dots, b_{n+1})$. This proves (1).
  Let $v\in \sigma _W[n]$ be one of the generators listed in (1).
  Then, we have
  \[
\pi (v) =
\begin{pmatrix}
-a_n + 1 \\
a_1 - a_n \\
\vdots \\
a_{n-1}-a_n \\
a_n
\end{pmatrix}
\]
If $a_n=0$, $\pi(v)$ is of the form
\[
\left(
\begin{array}{c}
  1 \\ \hdashline[2pt/2pt]
  u \\ \hdashline[2pt/2pt]
  0
\end{array}
\right)
\]
where $u$ is either a positive primitive weight vector $e_I$ or zero. If $a_n=1$,
\[
\pi (v)= \left(
\begin{array}{c}
  0 \\ \hdashline[2pt/2pt]
  u \\ \hdashline[2pt/2pt]
  1
\end{array}
\right)
\]
where $u$ is either a negative primitive weight vector $-e_I$ or zero.
Comparing with \eqref{sigma (n)},
we immediately conclude that $\pi (\sigma _W[n])=\sigma ^{(n)}$.
\end{proof}

Next, let us calculate the polytopal part $P_b[n]$ of $\tilde P_b[n]$,
which is given by
\[
P_b[n] = P_W[n] \cap (\alpha _W[n]\otimes \mathbb R)^{-1}(-b),
\]
for $b=b_{GHH}^{W[n]}= {}^t\!\left(\frac{n}{n+1},\frac{2n}{n+1},\dots ,\frac{n^2}{n+1}\right)$.
Recalling that $P_W[n]$ is the image of the hypercube $\square _{n^2}$ under the
linear map $L[n]$, first we look at
\[
\square _{n^2}\cap (\alpha _W[n] L[n])^{-1}(-b).
\]
We represent a vector in $\mathbb R^{n^2}$ by a transpose of
\[
\left(
\begin{array}{ccc|ccc|ccc|ccc}
  c_{11} & \cdots & c_{1n} & c_{21} & \cdots & c_{2n} & & \cdots & & c_{n1} & \cdots & c_{nn}
\end{array}
\right).
\]
As one can easily check that the composition $\alpha _W[n]L[n]$ is given by
$(n,n^2)$-matrix
\[
\alpha _W[n]L[n] =
\left(
 \begin{array}{ccc|ccc|ccc|ccc}
  -1 & \cdots & -1 & 0 & \cdots & 0 & & & & 0 & \cdots & 0 \\
  0 & \cdots & 0 & -1 & \cdots & -1 & & & & 0 & \cdots & 0 \\
  &&& &&& &\ddots & \\
  0 & \cdots & 0 & 0 & \cdots & 0 & & & & -1 & \cdots & -1
\end{array}
\right),
\]
the subspace $(\alpha _W[n] L[n])^{-1}(-b)$ is cut out by $n$ hyperplanes
\[
H_i: c_{i1}+c_{i2}+\dots +c_{in} = \frac{in}{n+1}\quad (i=1,\dots, n).
\]
If we define a polytope $R_i[n]$ in the subspace with coordinates $(c_{i1},c_{i2}, \dots, c_{in})$
as the intersection $\square _n\cap H_i$, we get a decomposition
\[
\square _{n^2}\cap (\alpha _W[n] L[n])^{-1}(-b)
= R_1[n]\times \dots \times R_n[n].
\]
As $R_i[n]$ is a hyperplane cut of
a hypercube $\square _n$, a vertex of $R_i[n]$ is on a edge of $\square _n$, that is,
a vertex is of the form $(c_{i1},\dots, c_{in})$ with $c_{ij}=0$ or $1$
for all $j$ but a unique $k$ and $0\leqslant c_{ik}\leqslant 1$.
Combining with the equation for $H_i$, one sees that the vertex set of
$R_i[n]$ is given by
\[
w_i = {}^t \! (
\underbrace{1,\dots ,1}_{\scriptsize (i-1)},\frac{n-i+1}{n+1},0,\dots ,0)
\]
and its permutation of components $sw_i\; (s\in \mathfrak S_n)$.
This implies that the polytope $P_b[n]$,
the image of $\square _{n^2}\cap (\alpha _W[n] L[n])^{-1}(-b)$ under $L[n]$,
is the convex hull of vectors
\[
\Bigg\{
L[n]
\left(
\begin{array}{c}
s_1w_1 \\ \hdashline[1pt/1pt]
\vdots \\ \hdashline[1pt/1pt]
s_nw_n
\end{array}
\right)\;\; \Bigg| \;\;
s_1,\dots ,s_n\in \mathfrak S_n
\Bigg\}.
\]
As we have
\begin{equation}\label{L[n] explicit}
L[n]
\left(
\begin{array}{c}
c_{11} \\ \vdots \\ c_{1n} \\ \hdashline[1pt/1pt]
\vdots \\ \hdashline[1pt/1pt]
c_{n1} \\ \vdots \\ c_{nn}
\end{array}
\right)
=
\left(
\begin{array}{c}
- \sum _{i=1} ^n c_{i1} \\ \vdots  \\
- \sum _{i=1} ^n  c_{in}  \\ \hdashline[1pt/1pt]
0 \\
\sum _{j=1} ^n c_{1j} \\
\sum _{j=1} ^n c_{1j} + \sum _{j=1} ^n c_{2j} \\
\vdots \\
\sum _{j=1} ^n c_{1j} + \cdots + \sum _{j=1} ^n c_{nj}
\end{array}
\right),
\end{equation}
the last $(n+1)$ entries of
$L[n]
\left(
\begin{array}{c}
s_1w_1 \\ \hdashline[1pt/1pt]
\vdots \\ \hdashline[1pt/1pt]
s_nw_n
\end{array}
\right)$
is always $\!\!\!\!\mbox{\phantom{$\Big($}}^t\!
\left(0, \frac{n}{n+1}, \frac{3n}{n+1},\dots, \frac{\frac 12 n^2(n+1)}{n+1}\right)$ independent of
permutations $s_1,\dots, s_n\in \mathfrak S_n$. As the kernel of $\alpha _W[n]$ is a free $\mathbb Z$-module
with a basis
\[
\left(
\begin{array}{c}
1 \\ \vdots \\ 0 \\ \hdashline[1pt/1pt]
0 \\ \vdots \\ 0
\end{array}
\right),\; \cdots \;,
\left(
\begin{array}{c}
0 \\ \vdots \\ 1 \\ \hdashline[1pt/1pt]
0 \\ \vdots \\ 0
\end{array}
\right),\;
\left(
\begin{array}{c}
0 \\ \vdots \\ 0 \\ \hdashline[1pt/1pt]
1 \\ \vdots \\ 1
\end{array}
\right),
\]
it is sufficient to look at the image of $P_b[n]$ under a projection to
first $n$ components. The first $n$ components of
$L[n]
\left(
\begin{array}{c}
w_1 \\ \hdashline[1pt/1pt]
\vdots \\ \hdashline[1pt/1pt]
w_n
\end{array}
\right)$ is given by
\[
u = \!\!\mbox{\phantom{$\Big($}}^t\! \left(
-\frac{n^2+n-1}{n+1},
-\frac{n^2-3}{n+1},
\dots,
-\frac{n+3}{n+1},
-\frac{1}{n+1},
\right).
\]
Here we note that the difference of every two consecutive numbers is
$1+\frac{1}{n+1}$.

\begin{lemma}\label{vertex of P_b[n]}
The projection $\bar P_b[n]$ of $P_b[n]$ to the first $n$ components is the convex hull of
the set $\{ su \; |\; s\in \mathfrak S_n\}$. In particular, $P_b[n]$ agrees with the permutahedron
$P^{(n)}$ up to a multiplication of rational scalar and a translation.
\end{lemma}

\begin{proof}
It is sufficient to show that
$L[n]
\left(
\begin{array}{c}
s_1w_1 \\ \hdashline[1pt/1pt]
\vdots \\ \hdashline[1pt/1pt]
s_nw_n
\end{array}
\right)$ is in the convex hull of $\{su\; |\; s\in \mathfrak S_n\}$ for
every $(s_1,\dots, s_n)\in (\mathfrak S_n)^n$.
By symmetry under the diagonal action of $\mathfrak S_n$,
this is equivalent to say that
$L[n]
\left(
\begin{array}{c}
s_1w_1 \\ \hdashline[1pt/1pt]
\vdots \\ \hdashline[1pt/1pt]
s_nw_n
\end{array}
\right)$
is in a cone $C$ spanned by \{$(i\; i+1)u\; |\; i=1,\dots, n-1\}$ with the vertex
$u$. One can easily check that the cone $C$ is defined by
\[
\begin{aligned}
a_1+\dots +a_k &\geqslant -\frac{k^2(n+2)-k(2n^2+3n)}{2(n+1)}\mbox{ for }k=1,\dots, n-1,\mbox{ and}\\
a_1+\dots +a_n &= -\frac{n^2}2,
\end{aligned}
\]
where $a_1,\dots, a_n$ are the first $n$ coordinates of $M_W[n]_{\mathbb R}$
as in Lemma \ref{conical part}. As $w_i$ satisfies
\[
c_{i1}\geqslant c_{i2}\geqslant \dots \geqslant c_{in},
\]
for all $i$, the sum $a_1+\dots +a_k=- \sum _{i=1}^n\sum _{j=1}^k c_{ij}$ only increases
under the action of $(s_1,\dots, s_n)\in (\mathfrak S_n)^n$ on
$\left(
\begin{array}{c}
w_1 \\ \hdashline[1pt/1pt]
\vdots \\ \hdashline[1pt/1pt]
w_n
\end{array}
\right)$.
\end{proof}

Now we finish the proof of Theorem \ref{quotiont of W[n]}.
On one hand $\tilde Z^{(n)'}$ is the toric variety corresponding to
a polyhedron $\tilde P^{(n)}=\sigma ^{(n)\vee}+\iota P^{(n)}$ by
Propoisiton \ref{polytope of Znp}. On the other hand,
the quotient $W[n]^{ss}\git G[n]$ is also a toric variety determined by
the polyhedron $\tilde P_b[n]$. By Lemma \ref{conical part},
the conical part of $\sigma '[n]$ coincides with $\sigma ^{(n)\vee}$
under the basis of $M'[n]$ consisting of column vectors of the transpose ${}^t\pi$ of
$\pi$ in \eqref{matrix pi}. Let us denote by $Q'$ the matrix of basis change
such that
\[
{}^t\pi Q' =
\left(
\begin{array}{ccc;{2pt/2pt}c}
&&& 0 \\
& I_n & & \vdots \\
&&& 0 \\ \hdashline[2pt/2pt]
&&& 1 \\
& O_{n+1} & & \vdots \\
&&& 1
\end{array}
\right).
\]
Then, we have
\[
\iota P^{(n)} = \frac{n+1}{n+2} \cdot Q' (P_b[n] - u).
\]
This implies that $\tilde P_b[n]$ agrees with the polyhedron $\tilde P^{(n)}$
after a multiplication of rational scalar $\frac{n+1}{n+2}$ and a translation.
Therefore, after taking sufficinetly high trancation of the graded rings,
we get an isomorphism $Z^{(n)\prime}=X(\tilde P^{(n)})\cong X(\tilde P_b[n])
= W[n]^{ss}\git G[n]$.

\begin{remark}\label{recovering original family}
In the same way, one can easily verfy that
$X[n]\git G[n]$ is isomorphic to the original family $X$.
In this case, the polytopal part $P_b$ of the quotient becomes
just one point so that the polyhedron $\tilde P_b$ is just
a cone that corresponds to the invariant ring of $X\times _{\mathbb A^1} \mathbb A^{n+1}$.
\end{remark}

\section{Hilbert-Chow morphism for GHH degeneration}

\subsection{GHH degeneration of Hilbert schemes}

Let us define $S=X\times \mathbb A^{m-1}$ and endow it with a morphism
$S\to C=\mathbb A^1$ given by the composition
\[
S\overset{\scriptsize \pr _1}{\lto}X\lto C=\mathbb A^1.
\]
This is a local model for a \emph{simple degeneration} in the sense of
\cite{GHH}, Definition 1.1. The expanded degeneration of the family $S\to C$ is just
given by the composition
\[
S[n] := X[n]\times \mathbb A^{m-1} \to  X[n]\to \mathbb A^{n+1}.
\]
Here we remark that the torus $G[n]$ acts trivially on the factor $\mathbb A^{m-1}$ in $S$ or $S[n]$.
Gulbrandsen, Halle, and Hulek considered in \emph{op. cit.}
the relative Hilbert scheme
\[
\Hilb ^n(S[n]/\mathbb A^{n+1})\to \mathbb A^{n+1}
\]
of the expanded degeneration $S[n]\to \mathbb A^{n+1}$.
Since the Hilbert scheme admits a natural action of $G[n]$, they define
\[
I^n_{S/C} = \Hilb ^n(S[n]/\mathbb A^{n+1})^{ss} \git G[n],
\]
where the GIT stability and the GIT quotient are considered
under the GHH linearization as in
\S\ref{GHH linearization}. $I^n_{S/C}$ has a natural morpshim
\[
I^n_{S/C}\to \mathbb A^{n+1}\git G[n]\cong \mathbb A^1,
\]
whose general fiber over $t\in \mathbb A^1$ is isomorphic to $\Hilb ^n(S_t)$, where
$S_t$ is the fiber over $t$ of the original semistable family $S\to \mathbb A^1$.
Let us call the family $I^n_{S/C}\to \mathbb A^1$
\emph{Gulbrandsen-Halle-Hulek degeneration}, or
\emph{GHH degeneration} of
Hilbert schemes associated with the family $S\to C$.

This construction is most interesting in the case where $m=2$, namely in the case where
$S\to C$ is a semistable family of surfaces whose singular fiber has no triple point.
As $\Hilb ^n(S[n]/\mathbb A^{n+1})\to \Sym ^n(S[n]/\mathbb A^{n+1})$ is $G[n]$-equivarinant,
we get a projective birational morphism
\[
\Psi : I^n_{S/C} \to \Sym^n (S[n]/\mathbb A^{n+1})^{ss} \git G[n].
\]
We call the morphism $\Psi$ \emph{Hilbert-Chow moprhism} of GHH degeneration.

\subsection{A small partial resolution $Z^{(n)}$ of $\Sym^n(S/C)$}
To analyze a degeneration of Hilbert schemes, we have another approach, namely
we can also start from the symmetric product
$\Sym ^n(S/C)$ of the family $S\to C$. As $S=X\times \mathbb A^1$,
it is just an $\mathfrak S_n$-quotient of $(S/C)^n = \tilde X^{(n) \prime} \times \mathbb A^n$.
The projective toric small resolution $\tilde Z^{(n)\prime}\to \tilde X^{(n)\prime}$
as in Proposition \ref{small resolu} immediately gives a toric small resolution
\[
\tilde Z^{(n)}=\tilde Z^{(n)\prime}\times \mathbb A^{n} \to \tilde X^{(n)\prime}\times \mathbb A^{n}=(S/C)^n.
\]
We note that the resolution is $\mathfrak S_n$-equivariant. By taking $\mathfrak S_n$-quotient of
both sides, we get a projective small resolution
\[
Z^{(n)}:=\tilde Z^{(n)}/\mathfrak S_n \to \Sym ^n(S/C).
\]
The self-product $(S[n]/\mathbb A^{n+1})^n$
of the expanded degeneration $S[n]=X[n]\times \mathbb A^1\to \mathbb A^{n+1}$
is just
\[
(S[n]/\mathbb A^{n+1})^n = W[n]\times \mathbb A^n\overset{\pr _1}{\lto} W[n]\lto \mathbb A^{n+1}.
\]
Again the torus $G[n]$ acts trivially on the factor $\mathbb A^n$,
and Theorem \ref{quotiont of W[n]} gives an isomorphism
\[
\tilde \varepsilon ^{(n)}: (W[n]\times \mathbb A^n)^{ss} \git G[n]\overset{\sim}{\lto}
\tilde Z^{(n)}=\tilde Z^{(n)\prime}\times \mathbb A^n.
\]
Moreover the $G[n]$-action on $W[n]\times \mathbb A^n$ commutes with the natural
$\mathfrak S^n$-action, $\tilde \varepsilon ^{(n)}$ descends to an isomorphism
\[
\varepsilon ^{(n)}: \Sym ^n(S[n]/\mathbb A^{n+1})^{ss}\git G[n] \to Z^{(n)},
\]
and thus we have a natural birational morphism
\[
\psi ^{(n), GHH} = \varepsilon ^{(n)}\circ \Psi :
I^n_{S/C}\to Z^{(n)}.
\]
Here we note that there is a commutative diagram
\begin{equation}\label{double quotients}
\xymatrix{
(W[n]\times \mathbb A^n)^{ss} \ar[r] \ar[d]
& \Sym^n(S[n]/\mathbb A^{n+1})^{ss} \ar[d]\\
\tilde Z^{(n)} \ar[r] &Z^{(n)} \\
}
\lower42pt\hbox{.}
\end{equation}

\subsection{A crepant resolution of $Z^{(n)}$}
By the construction, the general fiber of $Z^{(n)}\to C=\mathbb A^1$ is just
the symmetric product of the general fiber $\Sym ^n(S_t)$. More precisely,
as the restriction of the family $S\to C$ to $C^{\circ}=C\backslash \{0\}$ is a trivial
family of $\mathbb C^*\times \mathbb A^1$, the restriction
\[
Z^{(n)\circ}:=Z^{(n)}\times _C C^{\circ}\to C^{\circ}
\]
is a trivial family of $\Sym ^n(\mathbb C^*\times \mathbb A^1)$. Therefore,
the ordinary Hilbert-Chow morphism gives a crepant divisorial resolution
\[
\psi ^{(n)\circ} :Y^{(n)\circ} \to Z^{(n)\circ}.
\]
In \cite{N}, Theorem 4.1, we constructed an extension of $\psi ^{(n)\circ}$ to
a projective crepant divisorial birational morphism
\[
\psi ^{(n)}:Y^{(n)}\to Z^{(n)}.
\]
For $Z\in \Hilb ^n(S[n]/\mathbb A^{n+1})$, we define $t_i(Z)$ to be the
$i$-the component of the image of $Z$ in $\mathbb A^{n+1}$.
Since $Z\in \Hilb ^n(S[n]/\mathbb A^{n+1})$ has trivial stabilizer under
the action of $G[n]$ if $t_i(Z)\neq 0$ for all $i$,
we see that the restriction of $\psi ^{(n), GHH}$ agrees with the trivial
family of Hilbert-Chow morphisms $\psi ^{(n)\circ}$.
In the rest of the article, we prove the following

\begin{theorem}\label{comparison theorem}
The Hilbert-Chow morphism of GHH degeneration $\Psi$ (or $\psi ^{(n),GHH}$)
is isomorphic to the projective crepant divisorial partial resolution $\psi ^{(n)}$.
\end{theorem}

\subsection{Orbifold structures}\label{orbifold structures}
The key to prove the theorem is natural orbifold structures on
$I^n_{S/C}$ and $Y^{(n)}$. The orbifold structure on $I^n_{S/C}$ is explained in
\cite{GHH}, \S3: by the numerical criterion of stability (\cite{GHH}, Theorem 2.9),
a semistable point in $\Hilb ^n(S[n]/\mathbb A^{n+1})$ is automatically stable
under the GHH linearlizaiton,
and therefore the stabilizer subgroup of $G[n]$ at a point is finite.
Since the semistable locus $\Hilb ^n(S[n]/\mathbb A^{n+1})^{ss}$ is
contained in the smooth locus of $\Hilb ^n(S[n]/\mathbb A^{n+1})\to \mathbb A^{n+1}$
(\cite{GHH}, Lemmta 3.6 and 3.7),
Luna's \'etale slice theorem implies that the quotient stack
\[
\mathscr I^n_{S/C}=[\Hilb ^n(S[n]/\mathbb A^{n+1})^{ss} / G[n]]
\]
is a smooth Deligne-Mumford stack.
The canonical morphism to the coarse moduli scheme $\mathscr I^n_{S/C}\to I^n_{S/C}$ gives
the orbifold structure on $I^n_{S/C}$.
Similarly, the GIT quotient $\Sym ^n(S[n]/\mathbb A^{n+1})^{ss}\git G[n]$ has a natural covering
structure coming from the corresponding quotient stack $[\Sym ^n(S[n]/\mathbb A^{n+1})^{ss}/ G[n]]$.
Again by the numerical criterion of stability (\emph{loc. cit.}, see also the following remark),
one sees that the stack is also (non-smooth) Deligne-Mumford stack.

\begin{remark}\label{remark on numerical criterion}
Although the numerical criterion of stability \cite{GHH}, Theorem 2.9 is stated
only for a point in the Hilbert scheme $\Hilb ^n(S[n]/\mathbb A^{n+1})$,
one can easily verify that the proof
works in exactly the same way for the case of
the symmetric product $\Sym ^n(S[n]/\mathbb A^{n+1})$ and a self-fiber product $(S[n]/\mathbb A^{n+1})^n$.
We will repeatedly use this fact.
\end{remark}

On the other hand, $Y^{(n)}$ also has a natural orbifold structure. We recall the
construction of $Y^{(n)}$ in some detail (we refer \cite{N}, in particular
\S 2.7, Lemma 2.9, Lemma 4.5, and \S 4.6, for full detail).
Let us take a point $q\in Z^{(n)}$ and $\tilde q=(\tilde q_1,\tilde q_2)\in \tilde Z^{(n)}
=\tilde Z^{(n)\prime}\times \mathbb A^n$, a point above $q$. We have a sequence of
$\mathfrak S_n$-equivariant projections
\[
Z^{(n)\prime}=X(\Delta ^{(n)})\overset{\mbox{\scriptsize $\mathbb A^2$-bundle}}{\lto}
 X(A_{n-1})\overset{\mbox{\scriptsize birational}}{\lto}
 \mathbb P^{n-1}.
\]
If we denote the homogeneous coordinate on $\mathbb P^{n-1}$ by $[\xi _1:\dots :\xi _n]$
and take the coordinate $\displaystyle \left(\frac{\xi_1}{\xi _n},\dots, \frac{\xi _{n-1}}{\xi _n}\right)$
for the torus $(\mathbb C^*)^{n-1}\subset \mathbb P^{n-1}$,
an toric affine open neighborhood $X(\bar \delta ^{(n)})$ of $X(A_{n-1})$ is $\mathbb A^{n-1}$ with
the cooridantes
\[
\left(
\frac{\xi_1}{\xi_2},\frac{\xi_2}{\xi_3},\dots ,
\frac{\xi_{n-1}}{\xi_n}
\right)
\]
by \eqref{edge cone}, and therefore the toric coordinate on $x(s\bar \delta ^{(n)})
; (s\in \mathfrak S_n)$ is given
\[
\left(
\frac{\xi _{s(1)}}{\xi _{s(2)}},\frac{\xi _{s(2)}}{\xi _{s(3)}},\dots ,
\frac{\xi _{s(n-1)}}{\xi _{s(n)}}
\right)
\]
for some $s\in \mathfrak S_n$. Let us decompose $s\in \mathfrak S_n$ into cycles as
\[
s = (i_1\; \dots\; i_{l_1})(i_{l_1+1}\; \dots \; i_{l_2})\dots
(i_{l_{r-1}+1}\; \dots \; i_{l_r}),
\]
If $\tilde q_1$ is fixed by $s$, we necessarily have
\[
\frac{\xi _{i_{l_{k-1}+1}}}{\xi _{i_{l_{k-1}+2}}} =
\frac{\xi _{i_{l_{k-1}+2}}}{\xi _{i_{l_{k-1}+3}}} = \cdots =
\frac{\xi _{i_{l_k-1}}}{\xi _{i_{l_k}}} = \alpha _k
\]
for $k=1,\dots, r$, where $l_0=0$ by convention and
$\alpha_k$ is an $(l_k-l_{k-1})$-th root of unity.
We say that $\tilde q$ is an $s$-fixed point of \emph{trivial angle type} if
\[
\alpha _1 = \dots = \alpha _r = 1.
\]
One can easily see that $\tilde q$ is an $s$-fixed point of trivial angle if and only if
\[
\frac{\xi _i}{\xi _{s(i)}}=1 \quad \mbox{for all $i$ with $s(i)\neq i$}
\]
and $\tilde q_2\in \mathbb A^n$ is an $s$-fixed point with respect to the standard
permutation action. From this characterization, one sees that
\begin{equation}\label{stab of trivial angle type}
\Stab _{\mathfrak S_n}^0(\tilde q)=\{s\in \Stab _{\mathfrak S_n}(\tilde q)\; |\;
\mbox{$\tilde q$ is an $s$-fixed point of trivial angle type} \}
\end{equation}
is a Young subgroup of $\mathfrak S_n$ and is a normal subgroup of $\Stab _{\mathfrak S_n}(\tilde q)$
(\cite{N}, Lemma 4.5). The tangent space
$T_{\tilde q_1}\tilde Z^{(n)\prime}\cong \mathbb C^{n+1}$ seen as a representation of
$\Stab _{\mathfrak S_n}^0(\tilde q)$
is a direct sum of the restriction of standard permutation representation and a one dimensional trivial
representation. Therefore, for a sufficiently small neighborhood
$\tilde U_{\tilde q}\subset \tilde Z^{(n)}$ of $\tilde q$,
the quotient $U_{\tilde q}=\tilde U_{\tilde q}/\Stab _{\mathfrak S_n}(\tilde q)$ is isomorphic to an open neighborhood
of $(\gamma, 0)\in \Sym ^{n}(\mathbb A^2)\times \mathbb A^1$,
where $\gamma = \sum \mu _i p_i \in \Sym ^{n}(\mathbb A^2)$ if $\Stab _{\mathfrak S_n}^0(\tilde q)$
is isomorphic to a Young subgroup $\mathfrak S_{\mu}$ associated with a partition $\mu=(\mu _i)$ of $n$.
Now restricting the Hilbert-Chow morphism
$\Hilb ^n(\mathbb A^2)\times \mathbb A^1\to \Sym ^n(\mathbb A^2)\times \mathbb A^1$ to $U_{\tilde q}$,
we get a crepant divisorial resolution
\[
\widehat U_{\tilde q} \to U_{\tilde q}=\tilde U_{\tilde q}/\Stab _{\mathfrak S_n}^0(\tilde q).
\]
As $\Stab _{\mathfrak S_n}^0(\tilde q)$ is a normal subgroup of $\Stab _{\mathfrak S_n}(\tilde q)$,
the quotient group
\begin{equation}\label{quotient stabilizer}
G(\tilde q)=\Stab _{\mathfrak S_n}(\tilde q)/\Stab _{\mathfrak S_n}^0(\tilde q)
\end{equation}
acts on $U_{\tilde q}$ and moreover the action lifts to $\widehat U_{\tilde q}$.
Nothing that $U_{\tilde q}/G(\tilde q)$ is isomorphic to a neighborhood of
the image $q\in Z^{(n)}$ of the point $\tilde q\in \tilde Z^{(n)}$,
we see that the quotients with canonical morphism
\[
\widehat U_{\tilde q}/G(\tilde q) \to U_{\tilde q}/G(\tilde q)\subset Z^{(n)}
\]
patch together along $Z^{(n)}$ to give a crepant partial resolution
\[
\psi ^{(n)}:Y^{(n)}\to Z^{(n)}.
\]
Therefore, the family of quotient stacks $\{[\widehat U_{\tilde q}/G(\tilde q)]\to Z^{(n)}\}$
defines a smooth Deligne-Mumford stack $\mathscr Y^{(n)}$ whose coarse moduli space
is $Y^{(n)}$.

\subsection{Semistable locus $W[n]^{ss}$ and the quotient map}
In this subsection, we describe explicitly the local behavior of the quotient map
$W[n]^{ss}\to W[n]\git G[n]=\tilde Z^{(n)\prime}$.

The cone $\sigma [n]$ corresponding to the base change
$X(\sigma [n])= X\times _{\mathbb A^1}\mathbb A^{n+1}$ is generated by
the column vectors of $(n+2,2n)$-matrix
\[
\sigma [n] =
\left(
\begin{array}{ccccccc}
  0 & 1 & 0 & 1 & \cdots & 0 & 1 \\ \hdashline[2pt/2pt]
  1 & 1 & 0 & 0 & \cdots & 0 & 0 \\
  0 & 0 & 1 & 1 & \cdots & 0 & 0 \\
   \vdots & \vdots & \vdots & \vdots & \ddots & \vdots & \vdots  \\
  0 & 0 & 0 & 0 & \cdots & 1 & 1
\end{array}
\right).
\]
We can see this as follows;
let $\sigma _X$ be a cone generated by $\begin{pmatrix} 0 \\ 1\end{pmatrix}$ and
$\begin{pmatrix}1 \\ 1\end{pmatrix}$ in $N_{X,\mathbb R}=\mathbb R^2$,
and $\sigma _{\mathbb A^{n+1}}$
be the positive orthant in $N_{\mathbb A^{n+1},\mathbb R}=\mathbb R^{n+1}$.
Then, by \cite{N}, Lemma 1.5, $\sigma [n]$ is just a fiber product of
cones $\sigma _X\times _\mathbb R \sigma _{\mathbb A^{n+1}}$
under the maps
\[
(0\; 1):N_X=\mathbb Z^2\to \mathbb Z,\quad \mbox{and} \quad
(1\; 1\; \cdots \; 1) : N_{\mathbb A^{n+1}}=\mathbb Z^{n+1}\to \mathbb Z.
\]
Then the cone $\sigma _W[n]$ corresponding to the self-product
\[
X(\sigma[n])\times _{\mathbb A^{n+1}} \cdots \times _{\mathbb A^{n+1}} X(\sigma [n])
\]
is given by the fiber product $\sigma _W[n]$ of $n$-copies of the cone $\sigma [n]$ with respect to
the projection to the last $n$-factors, which is generated by the vectors
\[
v_{I,j} =
\left(
\begin{array}{c}
e_I \\ \hdashline[2pt/2pt]
e_j
\end{array}
\right)
\]
for $e_I = \sum _{i\in I} e_i \in \mathbb Z^n$ with a (possibly empty) subset
$I\subset \{1,\dots, n\}$
and $e_i\in \mathbb Z^n\; (i=1,\dots, n)$ the standard basis,
and similarly for $e_j\in \mathbb Z^{n+1}\; (j=1,\dots, n+1)$.
As $W[n]\to (X(\Sigma [n]))/\mathbb A^{n+1})^n$ is a toric small birational morphism,
the set of rays in the normal fan $\Sigma _W$ to the polyheron $\tilde P_W[n]$ coicides
with the set of rays generated by $v_{I,j}$. Therefore, the set of torus invariant divisors on $W[n]$
is in one to one correspondence with the set $\{v_{I,j}\}$. We denote by $D_{I,j}$
the torus invariant divisor corresponding to $v_{I,j}$.
$\tilde P_W[n]$ is cut out by
halfspaces defined by
\[
v_{I,j}\geqslant d_{I,j}
\]
for some $d_{I,j}\in \mathbb Z$,
where $v_{I,j}$ is seen as a linear functional on $M_W[n]$. The ample divisor corresponding to
the polyhedron $\tilde P_W[n]$ is given by
\[
D_{\tilde P_W[n]} = \sum (-d_{I,j})D_{I,j}.
\]
Since the functional
$v_{I,j}$ is non-negative on the conical part $\sigma _W[n]$, the constants $d_{I,j}$ is
defined by
\[
d_{I,j} = \min (\{ \langle v_{I,j} ,u\rangle \; | \; u \in P_W[n]\}\cup \{0\} ).
\]
As the polytopal part $P_W[n]$ is the image of hypercube $\square _{n^2}$ under $L[n]$,
we have
\[
d_{I,j} = \min ( \{ \langle {\,}^t L[n]v_{I,j},\tilde u\rangle \; | \; \tilde u\in \square _{n^2}\}
\cup \{0\}).
\]
On the other hand, the equality
\[
{}^t L[n] v_{I,j} =
\left(
\begin{array}{c}
\Vspc e_{I^c} \\
\Vspc \vdots \\
\Vspc e_{I^c} \\ \hdashline[2pt/2pt]
\Vspc -e_I \\
\Vspc\vdots \\
\Vspc -e_I
\end{array}
\right)
\begin{array}{@{\kern-\nulldelimiterspace}l@{}}
    \left.\begin{array}{@{}c@{}}\Vspc\\\Vspc\\\Vspc\end{array}\right\}\mbox{\scriptsize $j-1$} \\
    \left.\begin{array}{@{}c@{}}\Vspc\\\Vspc\\\Vspc\end{array}\right.
  \end{array}
\]
infers that
\[
d_{I,j} = -(n-j)\cdot \#(I).
\]
The space of sections of $\mathcal O(kD_{\tilde P_W[n]})$ is isomorphic to
the vector space spanned by the monomials $m\in k P_W[m]$.
Therefore the complete linear system $|kD_{\tilde P_W[n]}|$ consists of divisors
of the form
\[
\sum ( \langle v_{I,j},m \rangle - kd_{I,j}) D_{I,j}.
\]
Therefore, the subsystem of $G[n]$-invariant divisors is given by
\[
\Lambda _{b,k}=
\left\{
\sum ( \langle v_{I,j},m \rangle - kd_{I,j}) D_{I,j}\; |\;
m \in k P_b[n]\right\} .
\]

\begin{proposition}\label{unstable locus}
The stable base locus $\bigcap _k \mbox{\rm Bs}(\Lambda _{b,k})$, namely the locus of
unstable points with respect to the GHH-linearlization is
\[
W[n]\backslash W[n]^{ss} =
\bigcup _{\#(I)\neq j} D_{I,j}.
\]
\end{proposition}

\begin{proof}
It is sufficient to prove
\[
\langle v_{I,j},m \rangle - d_{I,j} \geqslant 0
\]
for every rational point $m\in P_b[n]$ and the equality is attained if and only if
$j=\#(I)$. It is equivalent to say that the same condition holds for every
vertex $m\in P_b[n]$. By Lemma \ref{vertex of P_b[n]}, a vertex of $P_b[n]$
is of the form
\[
m_s =
\left(
\begin{array}{c}
  su \\ \hdashline[2pt/2pt]
  0 \\
  \frac{n}{n+1} \\
  \frac{3n}{n+1} \\
  \vdots \\
  \frac{\frac 12 n^2(n+1)}{n+1}
\end{array}
\right)
\]
for $u=\!\!\mbox{\phantom{$\big($}}^t\!\left(-\frac{n^2+n-1}{n+1},
-\frac{n^2-3}{n+1},
\dots,
-\frac{n+3}{n+1},
-\frac{1}{n+1},
\right)$ and $s\in \mathfrak S_n$, we have
\[
\langle v_{I,j}, m_s\rangle - d_{I,j}= \sum _{i\in I} (su)_i + \frac{j(j+1)}2 \frac{n}{n+1} - \#(I)\cdot (n-j).
\]
Its minimum is given by
\begin{multline*}
\sum _{i=1}^{\#(I)} -\frac{(n^2+n-1)-(i-1)(n+2)}{n+1} + \frac{j(j+1)}2 \frac{n}{n+1} - \#(I)\cdot (n-j)\\
= \frac 1{2(n+1)} (j-\#(I))((j-\#(I))n+n-2\cdot \#(I)).
\end{multline*}
It is elementary exercise to show that this amount is always non-negative and equals to zero
if and only if $j=\#(I)$.
\end{proof}

Let $\mathbf t = (t_1,\dots, t_{n+1})\in \mathbb A^{n+1}$ and $X[n]_{\mathbf t}$ the
fiber of the expanded degeneration of $X=\mathbb A^2\to \mathbb A^1,\; (x,y)\mapsto t=xy$.
We recall the description
of the fiber $X[n]_{\mathbf t}$ (see \cite{GHH}, Proposition 1.11 for detail).
Of course, if all the $t_i$'s are non-zero, then the fiber is
just $\mathbb C^*$.
The case $\mathbf t=\mathbf 0 = (0,\dots, 0)$ is the `most degenerate' case;
$X[n]_{\mathbf 0}$ consists of $(n+1)$
curves $\Delta ^0,\dots, \Delta ^{n+1}$ that form a straight tree.
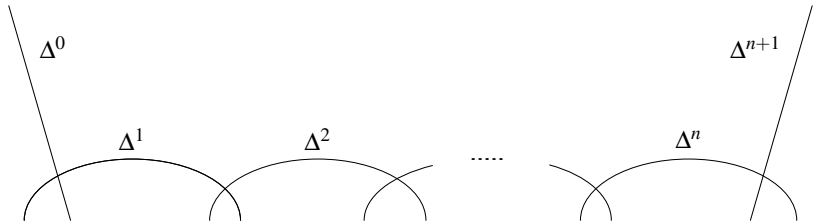
\begin{center}
\scalebox{.8}{
{\unitlength 0.1in%
\begin{picture}(52.0000,14.0000)(2.0000,-18.0000)%
%
\special{pn 8}%
\special{ar 1000 1800 700 400 3.1415927 6.2831853}%
%
\special{pn 8}%
\special{ar 1000 1800 700 400 3.1415927 6.2831853}%
%
\special{pn 8}%
\special{ar 2200 1800 700 400 3.1415927 6.2831853}%
%
\special{pn 8}%
\special{ar 3400 1800 700 400 5.1568082 6.2831853}%
%
\special{pn 8}%
\special{ar 4600 1800 700 400 3.1415927 6.2831853}%
%
\special{pn 8}%
\special{ar 3300 1800 800 400 3.1415927 4.2487414}%
%
\special{pn 13}%
\special{pa 3200 1400}%
\special{pa 3400 1400}%
\special{dt 0.045}%
%
\special{pn 8}%
\special{pa 600 1800}%
\special{pa 200 400}%
\special{fp}%
%
\special{pn 8}%
\special{pa 5000 1800}%
\special{pa 5400 400}%
\special{fp}%
\put(4.0000,-6.0000){\makebox(0,0)[lt]{$\Delta ^0$}}%
\put(10.0000,-12.7000){\makebox(0,0){$\Delta ^1$}}%
\put(22.0000,-12.7000){\makebox(0,0){$\Delta ^2$}}%
\put(46.0000,-12.7000){\makebox(0,0){$\Delta ^n$}}%
\put(52.0000,-6.0000){\makebox(0,0)[rt]{$\Delta ^{n+1}$}}%
\end{picture}}%
}
\end{center}
Here, $\Delta ^0$ and $\Delta ^{n+1}$ are $\mathbb A^1$ and
all the other $\Delta ^i$'s are $\mathbb P^1$. The intersection
$\Delta ^{i-1}\cap \Delta ^i$ is defined by $t_i=0$.
The intermediate case is a partial smoothing of the degeneration.
Let us define
\[
I_{\mathbf t} = \{i\; |\; t_i = 0\}.
\]
If $I_{\mathbf t}=\{i_1<i_2<\dots <i_r\}$, the fiber $X_{\mathbf t}$ consists of
$\Delta ^0,\Delta ^{i_1},\dots, \Delta ^{i_r}$ and
is a result of smoothing along the coordinate $t_i$ for $i\notin I_{\mathbf t}$
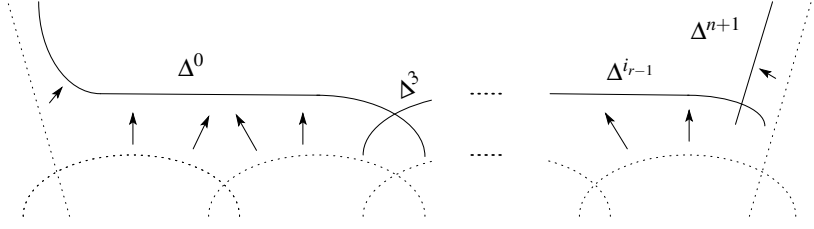
\begin{center}
  \scalebox{.8}{
{\unitlength 0.1in%
\begin{picture}(52.0000,14.0000)(2.0000,-18.0000)%
%
\special{pn 8}%
\special{pn 8}%
\special{pa 300 1800}%
\special{pa 300 1792}%
\special{fp}%
\special{pa 304 1755}%
\special{pa 306 1748}%
\special{fp}%
\special{pa 317 1713}%
\special{pa 320 1706}%
\special{fp}%
\special{pa 336 1672}%
\special{pa 341 1666}%
\special{fp}%
\special{pa 362 1635}%
\special{pa 367 1629}%
\special{fp}%
\special{pa 391 1602}%
\special{pa 397 1597}%
\special{fp}%
\special{pa 425 1572}%
\special{pa 431 1567}%
\special{fp}%
\special{pa 461 1545}%
\special{pa 467 1541}%
\special{fp}%
\special{pa 499 1521}%
\special{pa 505 1517}%
\special{fp}%
\special{pa 537 1499}%
\special{pa 545 1496}%
\special{fp}%
\special{pa 578 1481}%
\special{pa 586 1478}%
\special{fp}%
\special{pa 620 1464}%
\special{pa 628 1461}%
\special{fp}%
\special{pa 663 1450}%
\special{pa 671 1447}%
\special{fp}%
\special{pa 707 1437}%
\special{pa 714 1435}%
\special{fp}%
\special{pa 750 1426}%
\special{pa 758 1425}%
\special{fp}%
\special{pa 795 1418}%
\special{pa 802 1416}%
\special{fp}%
\special{pa 839 1411}%
\special{pa 847 1410}%
\special{fp}%
\special{pa 884 1405}%
\special{pa 892 1405}%
\special{fp}%
\special{pa 928 1402}%
\special{pa 936 1402}%
\special{fp}%
\special{pa 973 1400}%
\special{pa 981 1400}%
\special{fp}%
\special{pa 1019 1400}%
\special{pa 1027 1400}%
\special{fp}%
\special{pa 1064 1402}%
\special{pa 1072 1402}%
\special{fp}%
\special{pa 1109 1405}%
\special{pa 1117 1406}%
\special{fp}%
\special{pa 1154 1410}%
\special{pa 1161 1411}%
\special{fp}%
\special{pa 1198 1416}%
\special{pa 1206 1417}%
\special{fp}%
\special{pa 1242 1424}%
\special{pa 1250 1427}%
\special{fp}%
\special{pa 1286 1435}%
\special{pa 1294 1437}%
\special{fp}%
\special{pa 1330 1447}%
\special{pa 1337 1449}%
\special{fp}%
\special{pa 1372 1462}%
\special{pa 1380 1464}%
\special{fp}%
\special{pa 1414 1478}%
\special{pa 1422 1481}%
\special{fp}%
\special{pa 1455 1496}%
\special{pa 1463 1500}%
\special{fp}%
\special{pa 1495 1517}%
\special{pa 1502 1521}%
\special{fp}%
\special{pa 1533 1541}%
\special{pa 1539 1545}%
\special{fp}%
\special{pa 1569 1567}%
\special{pa 1575 1572}%
\special{fp}%
\special{pa 1603 1597}%
\special{pa 1609 1602}%
\special{fp}%
\special{pa 1634 1630}%
\special{pa 1638 1636}%
\special{fp}%
\special{pa 1660 1666}%
\special{pa 1664 1673}%
\special{fp}%
\special{pa 1681 1706}%
\special{pa 1683 1713}%
\special{fp}%
\special{pa 1694 1748}%
\special{pa 1695 1756}%
\special{fp}%
\special{pa 1700 1792}%
\special{pa 1700 1800}%
\special{fp}%
%
\special{pn 8}%
\special{pn 8}%
\special{pa 300 1800}%
\special{pa 300 1792}%
\special{fp}%
\special{pa 304 1755}%
\special{pa 306 1748}%
\special{fp}%
\special{pa 317 1713}%
\special{pa 320 1706}%
\special{fp}%
\special{pa 336 1672}%
\special{pa 341 1666}%
\special{fp}%
\special{pa 362 1635}%
\special{pa 367 1629}%
\special{fp}%
\special{pa 391 1602}%
\special{pa 397 1597}%
\special{fp}%
\special{pa 425 1572}%
\special{pa 431 1567}%
\special{fp}%
\special{pa 461 1545}%
\special{pa 467 1541}%
\special{fp}%
\special{pa 499 1521}%
\special{pa 505 1517}%
\special{fp}%
\special{pa 537 1499}%
\special{pa 545 1496}%
\special{fp}%
\special{pa 578 1481}%
\special{pa 586 1478}%
\special{fp}%
\special{pa 620 1464}%
\special{pa 628 1461}%
\special{fp}%
\special{pa 663 1450}%
\special{pa 671 1447}%
\special{fp}%
\special{pa 707 1437}%
\special{pa 714 1435}%
\special{fp}%
\special{pa 750 1426}%
\special{pa 758 1425}%
\special{fp}%
\special{pa 795 1418}%
\special{pa 802 1416}%
\special{fp}%
\special{pa 839 1411}%
\special{pa 847 1410}%
\special{fp}%
\special{pa 884 1405}%
\special{pa 892 1405}%
\special{fp}%
\special{pa 928 1402}%
\special{pa 936 1402}%
\special{fp}%
\special{pa 973 1400}%
\special{pa 981 1400}%
\special{fp}%
\special{pa 1019 1400}%
\special{pa 1027 1400}%
\special{fp}%
\special{pa 1064 1402}%
\special{pa 1072 1402}%
\special{fp}%
\special{pa 1109 1405}%
\special{pa 1117 1406}%
\special{fp}%
\special{pa 1154 1410}%
\special{pa 1161 1411}%
\special{fp}%
\special{pa 1198 1416}%
\special{pa 1206 1417}%
\special{fp}%
\special{pa 1242 1424}%
\special{pa 1250 1427}%
\special{fp}%
\special{pa 1286 1435}%
\special{pa 1294 1437}%
\special{fp}%
\special{pa 1330 1447}%
\special{pa 1337 1449}%
\special{fp}%
\special{pa 1372 1462}%
\special{pa 1380 1464}%
\special{fp}%
\special{pa 1414 1478}%
\special{pa 1422 1481}%
\special{fp}%
\special{pa 1455 1496}%
\special{pa 1463 1500}%
\special{fp}%
\special{pa 1495 1517}%
\special{pa 1502 1521}%
\special{fp}%
\special{pa 1533 1541}%
\special{pa 1539 1545}%
\special{fp}%
\special{pa 1569 1567}%
\special{pa 1575 1572}%
\special{fp}%
\special{pa 1603 1597}%
\special{pa 1609 1602}%
\special{fp}%
\special{pa 1634 1630}%
\special{pa 1638 1636}%
\special{fp}%
\special{pa 1660 1666}%
\special{pa 1664 1673}%
\special{fp}%
\special{pa 1681 1706}%
\special{pa 1683 1713}%
\special{fp}%
\special{pa 1694 1748}%
\special{pa 1695 1756}%
\special{fp}%
\special{pa 1700 1792}%
\special{pa 1700 1800}%
\special{fp}%
%
\special{pn 8}%
\special{pn 8}%
\special{pa 1500 1800}%
\special{pa 1500 1792}%
\special{fp}%
\special{pa 1504 1756}%
\special{pa 1506 1748}%
\special{fp}%
\special{pa 1517 1713}%
\special{pa 1519 1706}%
\special{fp}%
\special{pa 1536 1673}%
\special{pa 1540 1666}%
\special{fp}%
\special{pa 1562 1636}%
\special{pa 1566 1630}%
\special{fp}%
\special{pa 1591 1603}%
\special{pa 1596 1598}%
\special{fp}%
\special{pa 1624 1573}%
\special{pa 1630 1568}%
\special{fp}%
\special{pa 1660 1546}%
\special{pa 1666 1541}%
\special{fp}%
\special{pa 1698 1522}%
\special{pa 1704 1518}%
\special{fp}%
\special{pa 1737 1500}%
\special{pa 1744 1496}%
\special{fp}%
\special{pa 1778 1481}%
\special{pa 1785 1478}%
\special{fp}%
\special{pa 1820 1464}%
\special{pa 1827 1461}%
\special{fp}%
\special{pa 1863 1449}%
\special{pa 1870 1447}%
\special{fp}%
\special{pa 1906 1437}%
\special{pa 1914 1435}%
\special{fp}%
\special{pa 1950 1427}%
\special{pa 1958 1425}%
\special{fp}%
\special{pa 1994 1418}%
\special{pa 2002 1417}%
\special{fp}%
\special{pa 2038 1411}%
\special{pa 2046 1410}%
\special{fp}%
\special{pa 2083 1406}%
\special{pa 2091 1405}%
\special{fp}%
\special{pa 2128 1402}%
\special{pa 2136 1402}%
\special{fp}%
\special{pa 2173 1400}%
\special{pa 2181 1400}%
\special{fp}%
\special{pa 2219 1400}%
\special{pa 2227 1400}%
\special{fp}%
\special{pa 2264 1402}%
\special{pa 2272 1402}%
\special{fp}%
\special{pa 2309 1405}%
\special{pa 2317 1405}%
\special{fp}%
\special{pa 2354 1410}%
\special{pa 2362 1411}%
\special{fp}%
\special{pa 2398 1417}%
\special{pa 2406 1418}%
\special{fp}%
\special{pa 2442 1425}%
\special{pa 2450 1426}%
\special{fp}%
\special{pa 2486 1435}%
\special{pa 2494 1437}%
\special{fp}%
\special{pa 2529 1447}%
\special{pa 2537 1449}%
\special{fp}%
\special{pa 2573 1461}%
\special{pa 2580 1464}%
\special{fp}%
\special{pa 2615 1478}%
\special{pa 2622 1481}%
\special{fp}%
\special{pa 2656 1497}%
\special{pa 2663 1500}%
\special{fp}%
\special{pa 2695 1518}%
\special{pa 2702 1522}%
\special{fp}%
\special{pa 2733 1541}%
\special{pa 2740 1546}%
\special{fp}%
\special{pa 2770 1568}%
\special{pa 2776 1572}%
\special{fp}%
\special{pa 2803 1597}%
\special{pa 2809 1603}%
\special{fp}%
\special{pa 2834 1630}%
\special{pa 2838 1636}%
\special{fp}%
\special{pa 2860 1666}%
\special{pa 2864 1673}%
\special{fp}%
\special{pa 2880 1706}%
\special{pa 2883 1713}%
\special{fp}%
\special{pa 2894 1748}%
\special{pa 2896 1755}%
\special{fp}%
\special{pa 2900 1792}%
\special{pa 2900 1800}%
\special{fp}%
%
\special{pn 8}%
\special{pn 8}%
\special{pa 3701 1439}%
\special{pa 3709 1441}%
\special{fp}%
\special{pa 3743 1451}%
\special{pa 3751 1454}%
\special{fp}%
\special{pa 3785 1466}%
\special{pa 3792 1469}%
\special{fp}%
\special{pa 3825 1482}%
\special{pa 3832 1485}%
\special{fp}%
\special{pa 3864 1500}%
\special{pa 3871 1504}%
\special{fp}%
\special{pa 3902 1521}%
\special{pa 3909 1526}%
\special{fp}%
\special{pa 3939 1545}%
\special{pa 3946 1550}%
\special{fp}%
\special{pa 3975 1572}%
\special{pa 3981 1577}%
\special{fp}%
\special{pa 4008 1602}%
\special{pa 4013 1607}%
\special{fp}%
\special{pa 4037 1635}%
\special{pa 4042 1641}%
\special{fp}%
\special{pa 4062 1670}%
\special{pa 4066 1677}%
\special{fp}%
\special{pa 4081 1708}%
\special{pa 4084 1715}%
\special{fp}%
\special{pa 4094 1749}%
\special{pa 4096 1756}%
\special{fp}%
\special{pa 4100 1792}%
\special{pa 4100 1800}%
\special{fp}%
%
\special{pn 8}%
\special{pn 8}%
\special{pa 3900 1800}%
\special{pa 3900 1792}%
\special{fp}%
\special{pa 3904 1756}%
\special{pa 3906 1749}%
\special{fp}%
\special{pa 3916 1716}%
\special{pa 3918 1709}%
\special{fp}%
\special{pa 3933 1679}%
\special{pa 3936 1674}%
\special{fp}%
\special{pa 3955 1645}%
\special{pa 3959 1639}%
\special{fp}%
\special{pa 3981 1614}%
\special{pa 3986 1608}%
\special{fp}%
\special{pa 4010 1585}%
\special{pa 4015 1580}%
\special{fp}%
\special{pa 4042 1559}%
\special{pa 4047 1554}%
\special{fp}%
\special{pa 4076 1535}%
\special{pa 4083 1531}%
\special{fp}%
\special{pa 4113 1512}%
\special{pa 4120 1509}%
\special{fp}%
\special{pa 4151 1493}%
\special{pa 4158 1490}%
\special{fp}%
\special{pa 4191 1476}%
\special{pa 4198 1473}%
\special{fp}%
\special{pa 4232 1460}%
\special{pa 4239 1457}%
\special{fp}%
\special{pa 4272 1447}%
\special{pa 4279 1444}%
\special{fp}%
\special{pa 4314 1435}%
\special{pa 4322 1433}%
\special{fp}%
\special{pa 4357 1425}%
\special{pa 4364 1423}%
\special{fp}%
\special{pa 4399 1417}%
\special{pa 4406 1416}%
\special{fp}%
\special{pa 4441 1410}%
\special{pa 4449 1409}%
\special{fp}%
\special{pa 4485 1405}%
\special{pa 4493 1405}%
\special{fp}%
\special{pa 4529 1402}%
\special{pa 4537 1402}%
\special{fp}%
\special{pa 4573 1400}%
\special{pa 4581 1400}%
\special{fp}%
\special{pa 4618 1400}%
\special{pa 4626 1400}%
\special{fp}%
\special{pa 4662 1402}%
\special{pa 4670 1402}%
\special{fp}%
\special{pa 4706 1405}%
\special{pa 4714 1405}%
\special{fp}%
\special{pa 4750 1409}%
\special{pa 4758 1410}%
\special{fp}%
\special{pa 4793 1415}%
\special{pa 4801 1417}%
\special{fp}%
\special{pa 4836 1423}%
\special{pa 4843 1425}%
\special{fp}%
\special{pa 4878 1433}%
\special{pa 4886 1435}%
\special{fp}%
\special{pa 4920 1444}%
\special{pa 4927 1446}%
\special{fp}%
\special{pa 4961 1457}%
\special{pa 4968 1460}%
\special{fp}%
\special{pa 5001 1472}%
\special{pa 5008 1475}%
\special{fp}%
\special{pa 5041 1489}%
\special{pa 5048 1493}%
\special{fp}%
\special{pa 5079 1508}%
\special{pa 5086 1512}%
\special{fp}%
\special{pa 5117 1530}%
\special{pa 5123 1534}%
\special{fp}%
\special{pa 5152 1554}%
\special{pa 5158 1559}%
\special{fp}%
\special{pa 5184 1580}%
\special{pa 5190 1585}%
\special{fp}%
\special{pa 5215 1609}%
\special{pa 5220 1614}%
\special{fp}%
\special{pa 5241 1639}%
\special{pa 5245 1645}%
\special{fp}%
\special{pa 5264 1673}%
\special{pa 5268 1680}%
\special{fp}%
\special{pa 5282 1710}%
\special{pa 5285 1717}%
\special{fp}%
\special{pa 5294 1750}%
\special{pa 5296 1756}%
\special{fp}%
\special{pa 5300 1792}%
\special{pa 5300 1800}%
\special{fp}%
%
\special{pn 8}%
\special{pn 8}%
\special{pa 2500 1800}%
\special{pa 2500 1792}%
\special{fp}%
\special{pa 2505 1754}%
\special{pa 2507 1746}%
\special{fp}%
\special{pa 2520 1711}%
\special{pa 2523 1704}%
\special{fp}%
\special{pa 2542 1673}%
\special{pa 2546 1666}%
\special{fp}%
\special{pa 2570 1636}%
\special{pa 2576 1630}%
\special{fp}%
\special{pa 2603 1604}%
\special{pa 2610 1598}%
\special{fp}%
\special{pa 2639 1574}%
\special{pa 2646 1570}%
\special{fp}%
\special{pa 2678 1548}%
\special{pa 2685 1544}%
\special{fp}%
\special{pa 2718 1525}%
\special{pa 2726 1522}%
\special{fp}%
\special{pa 2760 1505}%
\special{pa 2767 1501}%
\special{fp}%
\special{pa 2803 1487}%
\special{pa 2810 1484}%
\special{fp}%
\special{pa 2846 1471}%
\special{pa 2854 1468}%
\special{fp}%
\special{pa 2890 1456}%
\special{pa 2898 1455}%
\special{fp}%
\special{pa 2934 1444}%
\special{pa 2942 1442}%
\special{fp}%
%
\special{pn 13}%
\special{pa 3200 1400}%
\special{pa 3400 1400}%
\special{dt 0.045}%
%
\special{pn 8}%
\special{pa 600 1800}%
\special{pa 200 400}%
\special{dt 0.045}%
%
\special{pn 8}%
\special{pa 5000 1800}%
\special{pa 5400 400}%
\special{dt 0.045}%
%
\special{pn 8}%
\special{ar 800 400 400 600 1.5707963 3.1415927}%
%
\special{pn 8}%
\special{ar 2200 1410 700 400 4.7123890 6.2831853}%
%
\special{pn 8}%
\special{ar 3300 1400 800 400 3.1415927 4.2487414}%
%
\special{pn 8}%
\special{pa 2220 1010}%
\special{pa 800 1000}%
\special{fp}%
%
\special{pn 13}%
\special{pa 3200 1000}%
\special{pa 3400 1000}%
\special{dt 0.045}%
%
\special{pn 8}%
\special{ar 4600 1210 500 200 4.7123890 6.2831853}%
%
\special{pn 8}%
\special{pa 4610 1010}%
\special{pa 3710 1000}%
\special{fp}%
%
\special{pn 8}%
\special{pa 4910 1200}%
\special{pa 5150 400}%
\special{fp}%
\put(13.9000,-8.1000){\makebox(0,0){$\Delta ^0$}}%
\put(27.9500,-9.3700){\rotatebox{24.2539}{\makebox(0,0){$\Delta ^3$}}}%
\put(42.3600,-8.4400){\makebox(0,0){$\Delta ^{i_{r-1}}$}}%
\put(49.5000,-5.0000){\makebox(0,0)[rt]{$\Delta ^{n+1}$}}%
%
\special{pn 8}%
\special{pa 5170 900}%
\special{pa 5070 840}%
\special{fp}%
\special{sh 1}%
\special{pa 5070 840}%
\special{pa 5117 891}%
\special{pa 5116 867}%
\special{pa 5137 857}%
\special{pa 5070 840}%
\special{fp}%
\special{pa 4610 1310}%
\special{pa 4610 1110}%
\special{fp}%
\special{sh 1}%
\special{pa 4610 1110}%
\special{pa 4590 1177}%
\special{pa 4610 1163}%
\special{pa 4630 1177}%
\special{pa 4610 1110}%
\special{fp}%
\special{pa 4210 1360}%
\special{pa 4070 1130}%
\special{fp}%
\special{sh 1}%
\special{pa 4070 1130}%
\special{pa 4088 1197}%
\special{pa 4098 1176}%
\special{pa 4122 1177}%
\special{pa 4070 1130}%
\special{fp}%
\special{pa 2110 1330}%
\special{pa 2110 1130}%
\special{fp}%
\special{sh 1}%
\special{pa 2110 1130}%
\special{pa 2090 1197}%
\special{pa 2110 1183}%
\special{pa 2130 1197}%
\special{pa 2110 1130}%
\special{fp}%
\special{pa 1810 1370}%
\special{pa 1680 1140}%
\special{fp}%
\special{sh 1}%
\special{pa 1680 1140}%
\special{pa 1695 1208}%
\special{pa 1706 1186}%
\special{pa 1730 1188}%
\special{pa 1680 1140}%
\special{fp}%
\special{pa 1400 1360}%
\special{pa 1500 1150}%
\special{fp}%
\special{sh 1}%
\special{pa 1500 1150}%
\special{pa 1453 1202}%
\special{pa 1477 1198}%
\special{pa 1489 1219}%
\special{pa 1500 1150}%
\special{fp}%
\special{pa 1010 1330}%
\special{pa 1010 1130}%
\special{fp}%
\special{sh 1}%
\special{pa 1010 1130}%
\special{pa 990 1197}%
\special{pa 1010 1183}%
\special{pa 1030 1197}%
\special{pa 1010 1130}%
\special{fp}%
\special{pa 470 1060}%
\special{pa 550 960}%
\special{fp}%
\special{sh 1}%
\special{pa 550 960}%
\special{pa 493 1000}%
\special{pa 517 1002}%
\special{pa 524 1025}%
\special{pa 550 960}%
\special{fp}%
\end{picture}}%
}
\end{center}
Again $\Delta ^0$ and $\Delta ^{i_r}$ are $\mathbb A^1$ and all the other components
are $\mathbb P^1$. Let $\Delta ^{i_l, \circ}= \Delta ^{i_l}\backslash \Sing (X_{\mathbf t})$.
We note that $\Delta ^{i_l, \circ}\cong \mathbb C^*$.

Let us take a cycle $\gamma = \sum m_i (p_i, p_i') \in \Sym ^n(S[n]/\mathbb A^{n+1})^{ss}$,
where $(p_i,p_i')$ is a point of $S[n]=X[n]\times \mathbb A^1$. We denote
the image of $\gamma$ under the projection by
\[
\mathbf t(\gamma) = (t_1(\gamma),\dots, t_{n+1}(\gamma)) \in \mathbb A^{n+1}.
\]
Then, $\gamma$ is supported on the fiber $X[n]_{\mathbf t(\gamma)}\times \mathbb A^1$.
The numerial criterion of stability (\emph{op. cit.}, Theorem 2.9, see also
\S4, (19)) imposes strong constraints on the
distribution of points; all $p_i$'s are in the smooth locus of $X[n]_{\mathbf t(\gamma)}$
and the degree of $\gamma _{|\Delta ^{i_l}\times \mathbb A^1}$ is $i_{l+1}-i_{l}$ (here we set
$i_0=1$ and $i_{r+1}=n+1$).

Let us take a point
\[
\tilde \gamma_1 = (p_1,\dots ,p_n)\in W[n]
= X[n]\times _{\mathbb A^{n+1}} \dots \times _{\mathbb A^{n+1}} X[n],
\]
namely $n$-tuple of points $p_j\in X[n]$ such that
\[
\mathbf t(p_1)=\dots =\mathbf t(p_n)\in \mathbb A^{n+1},
\]
where $\mathbf t(p_j)=(t_1(p_j),\dots, t_{n+1}(p_j))$ stands for
the image of $p_j$ in $\mathbb A^{n+1}$, as before.
Combining with a point $\tilde \gamma _2=(p_1',\dots ,p_n')\in \mathbb A^n$,
we specify a point
\[
\tilde \gamma =(\tilde \gamma _1,\tilde \gamma _2)\in W[n]\times \mathbb A^n.
\]
We recall that we have toric charts
$W_k\subset X[n]\; (k=1,\dots, n+1)$ defined by
\[
u_i\neq 0 \mbox{\; for $i<k$\; and } v_i\neq 0 \mbox{\; for $i\geqslant k$},
\]
(see \cite{GHH}, Remark 1.6). $W_k$ is isomorphic to $\mathbb A^{n+2}$ with toric coordinates
\[
\begin{aligned}
\big(& x,\frac{u_1}{v_1},t_2,\dots, t_{n+1}\big),  &(&k=1)\\
\big(& t_1,\dots, t_{k-1},\frac{v_{k-1}}{u_{k-1}},\frac{u_k}{v_k},
t_{k+1},\dots, t_{n+1}\big),\quad &(&1<k<n+1) \\
\big(& t_1,\dots, t_n,\frac{v_n}{u_n},y). &(&k=n+1)
\end{aligned}
\]
We also note that we have the relations
\begin{equation}\label{relations in W_is}
x\cdot \frac{u_1}{v_1}=t_1,\quad
\frac{v_{k-1}}{u_{k-1}}\frac{u_k}{v_k}=t_k,\quad \mbox{and}\quad
\frac{v_n}{u_n}\cdot y=t_{n+1}.
\end{equation}
If the image
\[
\gamma = \sum (p_i,p_i')\in \Sym ^n(S[n]/\mathbb A^{n+1})
\]
of $\tilde \gamma$ is semistable, by the stability criterion (\emph{op. cit.}, Theorem 2.9),
we may assume $p_i\in W_i$ after renumbering, and hence we may assume
\begin{equation}\label{renumbering}
\tilde \gamma _1=(p_1,p_2, \dots, p_n)\in W_1\times _{\mathbb A^{n+1}}
W_2\times _{\mathbb A^{n+1}}\cdots \times _{\mathbb A^{n+1}} W_n.
\end{equation}
We write the coordinate of $p_k$ as
\[
\begin{aligned}
p_1 &= \left(x_1,\frac{u_{11}}{v_{11}},t_2,\dots, t_{n+1}\right), \\
p_k &= \left(t_1,\dots, t_{k,k-1},\frac{v_{k,k-1}}{u_{k,k-1}},\frac{u_{k,k}}{v_{k,k}},
t_{k+1},\dots, t_{n+1}\right).\quad (1<k\leqslant n)
\end{aligned}
\]
Then, $W_1\times _{\mathbb A^{n+1}}\cdots \times _{\mathbb A^{n+1}}W_n$ is
an affine space with coordinate
\[
(w_1,\dots, w_n; w_{n+1},\dots, w_{2n}; w_{2n+1}) =
\left(\frac{u_{11}}{v_{11}},\dots ,\frac{u_{nn}}{v_{nn}};
x_1,\frac{v_{21}}{u_{21}},\dots, \frac{v_{n,n-1}}{u_{n,n-1}};
t_{n+1}
\right).
\]
As the right hand side of \eqref{coordinte of W[n]} equals to
\begin{multline*}
(\,(\, (x_1,y_1;\dots ;x_n,y_n),(t_1,\dots, t_{n+1})\, )\\
([u_{11}:v_{11}],\dots, [u_{1n}:v_{1n}]), \dots ,
([u_{n1}:v_{n1}],\dots, [u_{nn}:v_{nn}])\, ),
\end{multline*}
in our notation, where the image of $p_k$ under the map $X[n]\to X=\mathbb A^2$ is $(x_k,y_k)$,
we have
\begin{multline*}
(w_1,\dots, w_n; w_{n+1},\dots, w_{2n}; w_{2n+1})\\
=
\left(
\frac{s_1}{t_2\cdots t_{n+1}}, \frac{s_2}{t_3\cdots t_{n+1}}, \dots,
\frac{s_n}{t_{n+1}};
\frac{t_1\cdots t_{n+1}}{s_1},\frac{t_2\cdots t_{n+1}}{s_2},\dots ,
\frac{t_nt_{n+1}}{s_n};
t_{n+1}
\right).
\end{multline*}
The cooresponding cone of monomials $\sigma _{1\dots n}^{\vee}$
on $M_W[n]_{\mathbb R}$ is generated by
the column vectors of
\[
\sigma _{1\dots n}^{\vee} =
\left(
\begin{array}{cccc;{2pt/2pt}cccc;{2pt/2pt}c}
1 & 0 & \cdots & 0 & -1 & 0 & \cdots & 0 & 0 \\
0 & 1 & \cdots & 0 & 0 & -1 & \cdots & 0 & 0 \\
 & &\ddots & & & & \ddots & & \vdots  \\
0 & 0 & \cdots & 1 & 0 & 0 & \cdots & -1 & 0 \\ \hdashline[2pt/2pt]
0 & 0 & \cdots & 0 & 1 & 0 & \cdots & 0 & 0 \\
-1 & 0 & \cdots & 0 & 1 & 1 & \cdots & 0 & 0 \\
-1 & -1 & \cdots & 0 & & & \ddots & & \vdots  \\
& & \ddots && 1 & 1 &\cdots & 1 & 0 \\
-1 & -1 & \cdots & -1 & 1 & 1 & \cdots & 1 & 1
\end{array}
\right).
\]
One sees that its dual cone $\sigma _{1\dots n}$ is generated by the columns of
\[
\sigma _{1\cdots n}=
\left(
\begin{array}{ccccccccc}
0 & 1 & 1 & 1 & 1 & \cdots & 1 & 1 & 1 \\
0 & 0 & 0 & 1 & 1 & \cdots & 1 & 1 & 1 \\
&&&&& \ddots && \\
0 & 0 & 0 & 0 & 0 & \cdots & 0 & 1 & 1 \\ \hdashline[2pt/2pt]
1 & 1 & 0 & 0 & 0 & \cdots & 0 & 0 & 0 \\
0 & 0 & 1 & 1 & 0 & \cdots & 0 & 0 & 0 \\
&&&&& \ddots &&& \\
0 & 0 & 0 & 0 & 0 & \cdots & 1 & 1 & 0 \\
0 & 0 & 0 & 0 & 0 & \cdots & 0 & 0 & 1
\end{array}
\right),
\]
and
$X(\sigma _{1\dots n})= W_1\times _{\mathbb A^{n+1}}\cdots \times _{\mathbb A^{n+1}} W_n
\subset W[n]$. The affine quotient $X(\sigma _{1\dots n})\git G[n]$ is
an affine toric variety $X(\pi \sigma _{1\dots n})$ for $\pi$ defined in
\eqref{matrix pi}. A direct calculation immediately shows that
\[
\pi \sigma _{1\dots n} = \delta ^{(n)}
\]
and its dual cone is generated by the column vectors of
\begin{equation}\label{delta-n-dual2}
(\pi \sigma _{1\dots n})^{\vee} = \delta ^{(n)\,\vee}=
\begin{pmatrix}
1 & 0 & 0 & & & 0 & 0 \\
-1 & 1 & 0 & & & 0 & 0 \\
0 & -1 & 1 & & & 0 & 0 \\
0 & 0 & -1 & \ddots & & 0 & 0 \\
&&& \ddots & \ddots && \\
0 & 0 & 0 & & \ddots & 1 & 0 \\
0 & 0 & 0 & & & -1 & 0 \\
0 & 0 & 0 & & & 0 & 1
\end{pmatrix}.
\end{equation}
The columns corresponds to the invariant monomial functions $f_0,\dots, f_n$
that generates the coordinate ring of the quotient $X(\sigma _{1\dots n})\git G[n]$.
As we have
\[
{}^t \pi\, (\pi \sigma _{1\dots n})^{\vee}
=
\left(
\begin{array}{ccccc}
-1 & 1 & \cdots & 0 & 0 \\
0 & -1 & \cdots & 0 & 0 \\
\vdots & \vdots & & \vdots & \vdots \\
0 & 0 & \cdots & 1 & 0 \\
0 & 0 & \cdots & -1 & 0 \\
0 & 0 & \cdots & 0 & 1 \\ \hdashline[2pt/2pt]
1 & 0 & \cdots & 0 & 0 \\
\vdots & \vdots & & \vdots & \vdots \\
1 & 0 & \cdots & 0 & 0
\end{array}
\right),
\]
we get the relations
\begin{equation}\label{local quotient map}
\begin{aligned}
f_0 &= w_{n+1} = x_1,\\
f_k &= w_{k}w_{n+k+1} = \frac{u_{k,k}}{v_{k,k}}\frac{v_{k+1,k}}{u_{k+1,k}},
\quad (0<k<n)\\
f_n&=w_nw_{n+1} = \frac{u_{nn}}{v_{nn}}\cdot t_{n+1}.
\end{aligned}
\end{equation}
Among the generators of $\sigma _{1\cdots n}$, the one that corresponds to
an irreducible component of the locus of unstable points is of the form
\[
\left(
\begin{array}{c}
e_1+\dots +e_{j+1} \\ \hdashline[2pt/2pt]
e_j
\end{array}
\right)\quad (j=0,1,\dots, n-1)
\]
by Proposition \ref{unstable locus}, and the corresponding divisor $D_{\{1,\dots ,j+1\},j}$
is defined by $w_{j+1}=0$. Therefore, the locus of semistable points $X(\sigma _{1\dots n})^{ss}$
is given by $w_1\dots w_n \neq 0$. Summarizing everything up, we get the following

\begin{proposition}
Notation as above.
Let $\tilde \gamma =(\tilde \gamma_1,\tilde \gamma _2)\in X(\sigma _{1\dots n})^{ss}\times \mathbb A^n$,
and denote the value of the function $w_j$ at $\tilde \gamma _1$ by $w_j(\tilde \gamma _1)$.
Then the affine subspace of $X(\sigma _{1\dots n})\times \mathbb A^n$ defined by
\[
w_j=w_j(\tilde \gamma _1)\quad (j=1,\dots, n)
\]
gives a $\Stab _{\mathfrak S_n}(\tilde \gamma)$-invariant
slice $\tilde V_{\tilde \gamma}$ at $\tilde \gamma$ to the quotient map
\[
\xymatrix @C=20pt @R=2pt {
W[n]^{ss}\times \mathbb A^n \ar[rrr] & & & \tilde Z^{(n)} \\
\cup &&&\cup \\
X(\sigma _{1\dots n})^{ss}\times \mathbb A^n
\ar[rr]^{(f_0,\dots, f_{n};\; \id)\hspace{20pt}} & & (X(\sigma _{1\dots n})^{ss}\git G[n])\times \mathbb A^n
\ar@{=}[r]^{\hspace{28pt}\sim} &X(\delta ^{(n)})\times \mathbb A^n
}\lower42pt\hbox{,}
\]
namely, the quotient map restricted to $\tilde V_{\tilde \gamma}$ gives an isomorphism
$\tilde V_{\tilde \gamma}\overset{\sim}{\to} X(\delta ^{(n)})\times \mathbb A^n$.
\end{proposition}

\subsection{Comparison of stabilizer subgroups}
To prove Theorem \ref{comparison theorem}, we compare the Deligne-Mumford stacks
$\mathscr I^n_{S/C}$ and $\mathscr Y^{(n)}$.

\begin{lemma}\label{comparison of stab of trivial angle type}
 Let $\tilde \gamma =\left((p_1,p_1'),\dots ,(p_n,p_n')\right)
 \in X(\sigma _{1\dots n})^{ss}\times \mathbb A^n$ and
 $\tilde q\in X(\delta ^{(n)})\times \mathbb A^n\subset \tilde Z^{(n)}$ be its image.
 Recall that we defined the sequence
 $I_{\mathbf t(\tilde \gamma)}$ as
 \[
 I_{\mathbf t(\tilde \gamma)}=\{i\, |\, t_i(\tilde \gamma) = 0\}=\{i_1<\dots <i_r\}.
 \]
 and $i_0=1, i_{r+1}=n+1$ by convention.
 Then,
 \begin{enumerate}[\rm (1)]
    \item If $s\in \Stab _{\mathfrak S_n}(\tilde q)$ and $i_l\leqslant j<i_{l+1}$,
    we have $i_l\leqslant s(j)<i_{l+1}$.
    \item $\Stab _{\mathfrak S_n}(\tilde \gamma)=
    \Stab _{\mathfrak S_n}^0(\tilde q)$\, (see \eqref{stab of trivial angle type}).
 \end{enumerate}

\end{lemma}

\begin{proof}
(1) We may assume $j<s(j)$. $s\in \Stab _{\mathfrak S_n}(\tilde q)$ implies that
there exists a root of unity $\alpha$ such that
\[
\frac{\xi _j}{\xi _{s(j)}} = \alpha.
\]
As we have
\[
\left(\frac{\xi _1}{\xi _2},\frac{\xi _2}{\xi _3},\dots, \frac{\xi _{n-1}}{\xi _n}\right)
= (f_1,\dots, f_{n-1})
\]
in the coordinate ring of $X(\delta ^{(n)})$ by \eqref{delta-n-dual} and \eqref{delta-n-dual2}, we know
\begin{equation}\label{bridge product}
\begin{aligned}
  \alpha &=\frac{\xi _j}{\xi _{j+1}}\frac {\xi _{j+1}}{\xi _{j+2}}\dots \frac{\xi _{s(j)-1}}{\xi _{s(j)}}\\
  &=f_{j}f_{j+1}\dots f_{s(j)-1}\\
  &=\frac{u_{j,j}}{v_{j,j}}\frac{v_{j+1,j}}{u_{j+1,j}}\cdot
  \frac{u_{j+1,j+1}}{v_{j+1,j+1}}\frac{v_{j+2,j+1}}{u_{j+2,j+1}}\cdots
  \frac{u_{s(j)-1,s(j)-1}}{v_{s(j)-1,s(j)-1}}\frac{v_{s(j),s(j)-1}}{u_{s(j),s(j)-1}}.
\end{aligned}
\end{equation}
As we have
$0=t_{i_{l+1}}=\frac{v_{i_{l+1},i_{l+1}-1}}{u_{i_{l+1},i_{l+1}-1}}
\frac{u_{i_{l+1},i_{l+1}}}{v_{i_{l+1},i_{l+1}}}$ as in \eqref{relations in W_is},
if $s(j)\geqslant i_{l+1}$, the product \eqref{bridge product} must be zero,
which is a contradiction. \\
(2) It is sufficient to prove that $s\in \Stab_{\mathfrak S_n}^0(\tilde q)$ if and only if
$(p_{s(i)},p_{s(i)}')=(p_i,p_i')$ for every $i$. As $\mathfrak S_n$ acts on $W[n]\times \mathbb A^n$
by simultaneous permutations, $\Stab _{\mathfrak S_n}(\tilde \gamma)$ is a Young subgroup.
As we know that $\Stab _{\mathfrak S_n}^0(\tilde q)$ is also a Young subgroup,
we may assume that $s$ is a transposition $(i\; j)$ for $i<j$.
Then, $s\in \Stab _{\mathfrak S_n}^0(\tilde q)$
is equivalent to $\displaystyle \frac{x_i}{x_j}=1$ and $\tilde \gamma _2=\tilde q_2$ is $s$-invariant.
The first condition can be rewritten as
\[
f_i \cdot f_{i+1} \cdots f_{j-1} =
\frac{u_{i,i}}{v_{i,i}}\frac{v_{i+1,i}}{u_{i+1,i}}\cdot
\frac{u_{i+1,i+1}}{v_{i+1,i+1}}\frac{v_{i+2,i+1}}{u_{i+2,i+1}} \cdots
\frac{u_{j-1,j-1}}{v_{j-1,j-1}}\frac{v_{j,j-1}}{u_{j,j-1}} =1.
\]
Using the relations \eqref{relations in W_is}, one can further rephrase the condition as
\[
\begin{aligned}
1 &= \frac{u_{i,i}}{v_{i,i}}\cdot t_{i+1}\cdots t_{j-1}\cdot \frac{v_{j,j-1}}{u_{j,j-1}} \\
  &= \frac{u_{i,i}}{v_{i,i}}\left(\frac{v_{i,i}}{u_{i,i}}\frac{u_{i,i+1}}{v_{i,i+1}}\right)
  \cdots \left(\frac{v_{i,j-2}}{u_{i,j-2}}\frac{u_{i,j-1}}{v_{i,j-1}}\right) \frac{v_{j,j-1}}{u_{j,j-1}}\\
  & = \frac{u_{i,j-1}}{v_{i,j-1}}\frac{v_{j,j-1}}{u_{j,j-1}},
\end{aligned}
\]
which is clearly equivalent to $p_i=p_j$.
\end{proof}

Let us write $\Stab _{\mathfrak S_n}(\tilde \gamma) =\mathfrak S_{M(\tilde \gamma)}$,
where $\mathfrak S_{M(\tilde \gamma)}$ is a Young subgroup associated with a partition
$M(\tilde \gamma)=\{M(\tilde \gamma)_k\}$;
\[
\{1,\dots, n\}=\coprod _k M(\tilde \gamma)_k.
\]
Lemma \ref{comparison of stab of trivial angle type}, (1)
implies that by further renumbering of $\tilde \gamma$ staying inside
$X(\sigma _{1\dots n})= W_1\times _{\mathbb A^{n+1}}\cdots \times _{\mathbb A^{n+1}} W_n$,
we may assume that
\begin{equation}\label{convenient partition}
M(\tilde \gamma)_k = \{m_k,m_k+1,\dots, m_{k+1}-1\}
\end{equation}
for a sequence $1=m_1<m_2<\dots <m_{\nu}<m_{\nu+1}=n$ and
the partition $M(\tilde \gamma)$ is a sub-partition
of the partition determined by
$I_{\mathbf t(\tilde \gamma)}$. More precisely, for each $l$, we have a partition
\[
i_l = m_{k_l}<m_{k_l+1}<\dots <m_{k_l+\beta _l-1}<m_{k_l+\beta _l}=m_{k_{l+1}}=i_{l+1}.
\]
We prepare the following notation:
\begin{equation}\label{bunch of notations}
  \renewcommand{\arraystretch}{1.5}
  \begin{array}{lll}
  \multicolumn{2}{l}{\mu _k = m_{k+1} - m_k = \#(M(\tilde \gamma)_k),} \\
  \multicolumn{2}{l}{K_l = \{ k_l, k_l+1,\dots, k_l+\beta _l-1\},} \\
  K^d = \{k\, |\, \mu _k = d\}, & K^d_l = K_l\cap K^d, \\
  \multicolumn{2}{l}{K_l^* = \{ K^d_l \}_d\quad \mbox{a partition of $K_l$},} \\
  \mathbf M_l =\{ M(\tilde \gamma)_k\; | \; k\in K_l\}, \qquad &
  \mathbf M^d =\{ M(\tilde \gamma)_k\; |\; k\in K^d\}, \\
  \mathbf M^d_l = \mathbf M_l\cap \mathbf M_d. &
\end{array}
  \renewcommand{\arraystretch}{1}
\end{equation}

\begin{lemma}
  Notation as above.
  \begin{enumerate}[\rm (1)]
   \item $s\in \Stab _{\mathfrak S_n}(\tilde q)$
   induces a permutation of the set $\mathbf M_l^d$ for each $d$.
   \item There is an injective homomorphism
   \[
   \rho: G(\tilde q)=\Stab _{\mathfrak S_n}(\tilde q)/\Stab _{\mathfrak S_n}^0(\tilde q)
   \to \prod _l \mathfrak S_{K^*_l}
   \]
   where $\mathfrak S_{K^*_l}\subset \mathfrak SK_l$ is the Young subgroup
   associated with the partition $K^*_l$ of $K_l$.
\end{enumerate}
\end{lemma}

\begin{proof}
(1) Take a cyclic permutation $c_k = (m_k\; \cdots \; m_{k+1}-1)\in \mathfrak SM(\tilde \gamma)_k$
in accordance with \eqref{convenient partition} for each $k$.
Since an element $s\in \Stab _{\mathfrak S_n}(\tilde q)/\Stab _{\mathfrak S_n}^0(\tilde q)$
normalizes $\Stab _{\mathfrak S_n}^0(\tilde q)= \mathfrak S_{M(\tilde \gamma)}$,
$sc_ks^{-1}=(s(m_k)\; \cdots \; s(m_{k+1}-1))\in \mathfrak S_{M(\tilde \gamma)}$.
Therefore, there is $k'$ such that $\mu _{k'}\geqslant \mu _k$
and $s(m_k),\dots, s(m_{k+1}-1)\in M(\tilde \gamma)_{k'}$. By a similar argument for
$s^{-1}c_{k'}s$, we conclude that $\mu _k=\mu _{k'}$. This implies that $s$ induces a
permutation of the set $\mathbf M^d$ for each $d$.
Moreover, Lemma \ref{comparison of stab of trivial angle type}, (1) asserts that
this permutation leave $\mathbf M_l^d$ invariant. \\
(2) (1) implies that an element of $\Stab _{\mathfrak S_n}(\tilde q)$ induces
a permutation of the set $K^d_l$. Therefore, we have a natural map
$\tilde \rho: \Stab _{\mathfrak S_n}(\tilde q)\to \prod \mathfrak S_{K^*_l}$.
Taking (1) into account, it is straightforward to check that for $s,s'\in \Stab _{\mathfrak S_n}(\tilde q)$,
$\tilde \rho (s)=\tilde \rho (s')$ if and only if $s^{-1}s'\in \mathfrak S_{M(\tilde \gamma)}
=\Stab _{\mathfrak S_n}^0(\tilde q)$, and hence $\tilde q$ decends to
an injective homomorphism $\rho: G(\tilde q)\to \prod \mathfrak S_{K^*_l}$.
\end{proof}

\begin{lemma}\label{torus stabilizer group}
 Let $\gamma\in \Sym ^n(S[n]/\mathbb A^{n+1})^{ss}$ and assume
 $I_{\mathbf t(\gamma)}=\{i_1<\dots <i_r\}$. Let us
 denote the restriction of $\gamma$ to $\Delta ^{i_l,\circ}\times \mathbb A^1$ by $\gamma _l\; (l=0,\dots, r)$.
 Then, the stabilizer subgroup of $\gamma$ under the action of $G[n]$ is given by
 \[
 \Stab _{G[n]}(\gamma) = \prod _{1\leqslant l \leqslant r-1} \Stab (\gamma _l),
 \]
 where
 \[
\Stab (\gamma _l) = \{\tau \in \mathbb C^*\; |\; \tau \cdot \gamma _l=\gamma _l\}
 \]
 is the stabilizer subgroup where
 $\mathbb C^*$ acts on $\Delta ^{i_l, \circ}\times \mathbb A^1$ by multiplication on the first
 factor and trivially on the second factor.
\end{lemma}

\begin{proof}
Let us take $(\lambda _1,\dots, \lambda _{n+1})\in G[n]$ with $\lambda _1\dots \lambda _{n+1}=1$.
If $t_i(\gamma)\neq 0$, then $\lambda _i$ acts freely in the orbit $G[n]\cdot \gamma$.
Therefore, the stabilizer subgroup $\Stab _{G[n]}(\gamma)$ is actually a subgroup of
$(\mathbb C^*)^{r-1}$ with coordinate $(\lambda _{i_1},\dots, \lambda _{i_r})$ up to
the relation $\lambda _{i_1}\cdots \lambda _{i_r}=1$. Since $u_{i_l}/v_{i_l}$ gives a coordinate of
$\Delta ^{i_l,\circ}$,
if we introduce $\tau$-coordinate
\[
\tau _{i_l} = \lambda _{i_1}\cdots \lambda _{i_l}
\]
as in \S\ref{GHH linearization}, $\tau _{i_l}$ acts on $\Delta ^{i_l, \circ}$ by multiplication
and trivially on the other components $\Delta ^{i_j,\circ}$, from which the lemma immediately follows.
\end{proof}

Let us take a sufficiently small neighborhood $\tilde U_{\tilde q}$ of $\tilde q\in \tilde Z^{(n)}$ and
replace $\tilde V_{\tilde \gamma}$ by its inverse image so that the quotient map
$W[n]^{ss}\times \mathbb A^n\to \tilde Z^{(n)}$ restricts to an isomorphism
$\tilde V_{\tilde \gamma}\overset{\sim}{\to}\tilde U_{\tilde q}$.
Lemma \ref{comparison of stab of trivial angle type}, (2) asserts that
we have an induced isomorphism
\[
V_{\gamma} = \tilde V_{\tilde \gamma}/\Stab _{\mathfrak S_n}(\tilde \gamma)
\overset{\sim}{\lto} \tilde U_{\tilde q}/\Stab _{\mathfrak S_n}^0(\tilde q)=U_{\tilde q}.
\]
Here we note that $V_{\gamma}$ is identified with a slice at the image
$\gamma\in \Sym  ^n(S[n]/\mathbb A^{n+1})^{ss}$ of $\tilde \gamma$ with respect to
the quotient map $\Sym  ^n(S[n]/\mathbb A^{n+1})^{ss}\to Z^{(n)}$.

\begin{theorem}\label{thm: comparison of stabilizers}
The quotient map $\Sym^n(S[n]/\mathbb A^{n+1})^{ss}\to Z^{(n)}$ induces
an isomorphism
\[
\Stab _{G[n]}(\gamma)
\overset{\sim}{\lto}
G(\tilde q)
= \Stab _{\mathfrak S_n}(\tilde q)/ \Stab _{\mathfrak S_n}^0(\tilde q)
\]
and the isomorphism $V_{\gamma}\overset{\sim}{\lto} U_{\tilde q}$ is equivariant.
\end{theorem}

Before giving a proof of this theorem, let us look at a handy case.

\begin{example}
（1)\; Let $n=9$ and consider $\gamma =\sum _{i=1}^9 (p_i,p'_i)\in
\Sym ^9(S[9]/\mathbb A ^{10})^{ss}$.
Let us assume $I_{\mathbf t(\gamma)}=\{1<7<10\}$.
Then, the fiber $X[9]_{\mathbf t(\gamma)}$ consists of
4 components. By stablilty, six points are on $\Delta ^{1,\circ}\times \mathbb A^1$
and three points are on $\Delta ^{7,\circ}\times \mathbb A^1$.
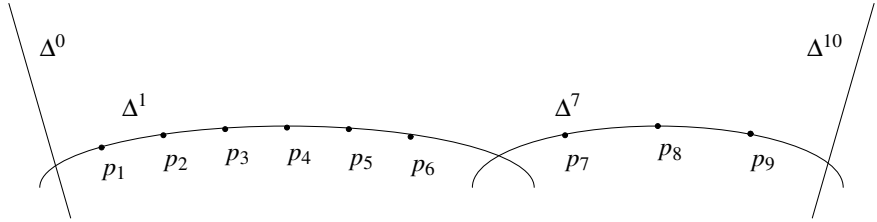
\begin{center}
\scalebox{.8}{
{\unitlength 0.1in%
\begin{picture}(56.0000,14.0000)(2.0000,-18.0000)%
%
\special{pn 8}%
\special{pa 600 1800}%
\special{pa 200 400}%
\special{fp}%
%
\special{pn 8}%
\special{pa 5400 1800}%
\special{pa 5800 400}%
\special{fp}%
\put(4.0000,-6.0000){\makebox(0,0)[lt]{$\Delta ^0$}}%
\put(10.2000,-10.6000){\makebox(0,0){$\Delta ^1$}}%
\put(56.0000,-6.0000){\makebox(0,0)[rt]{$\Delta ^{10}$}}%
%
\special{pn 8}%
\special{ar 2000 1600 1600 400 3.1415927 6.2831853}%
%
\special{pn 8}%
\special{ar 4400 1600 1200 400 3.1415927 6.2831853}%
\put(38.2000,-10.6000){\makebox(0,0){$\Delta ^7$}}%
%
\special{pn 4}%
\special{sh 1}%
\special{ar 800 1340 16 16 0 6.2831853}%
\special{sh 1}%
\special{ar 1200 1260 16 16 0 6.2831853}%
\special{sh 1}%
\special{ar 1600 1220 16 16 0 6.2831853}%
\special{sh 1}%
\special{ar 2000 1210 16 16 0 6.2831853}%
\special{sh 1}%
\special{ar 2400 1220 16 16 0 6.2831853}%
\special{sh 1}%
\special{ar 2800 1270 16 16 0 6.2831853}%
\special{sh 1}%
\special{ar 3800 1260 16 16 0 6.2831853}%
\special{sh 1}%
\special{ar 4400 1200 16 16 0 6.2831853}%
\special{sh 1}%
\special{ar 5000 1250 16 16 0 6.2831853}%
\special{sh 1}%
\special{ar 5000 1250 16 16 0 6.2831853}%
\put(8.0000,-14.5000){\makebox(0,0)[lt]{$p_1$}}%
\put(12.0000,-13.9000){\makebox(0,0)[lt]{$p_2$}}%
\put(16.0000,-13.5000){\makebox(0,0)[lt]{$p_3$}}%
\put(20.0000,-13.5000){\makebox(0,0)[lt]{$p_4$}}%
\put(24.0000,-13.8000){\makebox(0,0)[lt]{$p_5$}}%
\put(28.0000,-14.2000){\makebox(0,0)[lt]{$p_6$}}%
\put(38.0000,-13.9000){\makebox(0,0)[lt]{$p_7$}}%
\put(44.0000,-13.2000){\makebox(0,0)[lt]{$p_8$}}%
\put(50.0000,-13.9000){\makebox(0,0)[lt]{$p_9$}}%
\end{picture}}%
}
\end{center}
Let $\zeta _1=\frac{u_{11}}{v_{11}}$
and $\zeta _4=\frac{u_{41}}{v_{41}}$, the coordinate of $p_1$ and $p_4$,
respectively, and $\zeta _7=\frac{u_{77}}{v_{77}}$, the coordinate of $p_7$.
Let $\omega $ be a primitive third root of unity, and assume
\[
\begin{aligned}
  p_2&= \omega p_1,\quad p_3 = \omega^2 p_1, &p'_1&=p'_2=p'_3\\
  p_5&= \omega p_4,\quad p_6 = \omega^2 p_4, &p'_4&=p'_5=p'_6\\
  p_8&= \omega p_7,\quad p_9 = \omega^2 p_7, &p'_7&=p'_8=p'_9
\end{aligned}
\]
Furthermore, we assume that $(p_1,p_1')$ and $(p_4, p_4')$
are in general position,
namely $p_1\neq p_4$ or
the ratio $\zeta _1/\zeta _4$ is not a sixth root of unity.
Then by Lemma \ref{torus stabilizer group},
$\Stab _{G[n]}(\gamma)\cong (\mathbb Z/3\mathbb Z)^2$,
which is generated by multiplications of $\omega$ on
$\Delta ^{1,\circ}$ and $\Delta ^{7,\circ}$.
Now we calculate the value of the invariant functions $f_k$.
Since $t_1=t_{10}=0$, we have $f_0=f_9=0$, while
$\displaystyle f_k=\frac{u_{kk}}{v_{kk}}\frac{v_{k+1,k}}{u_{k+1,k}}$,
$\displaystyle \frac{v_{k,j-1}}{u_{k,j-1}}\frac{u_{k,j}}{v_{k,j}}=t_j(\gamma)$,
and $t_7=0$
imply that
\[
f_1=f_2=f_4=f_5=f_7=f_8=\omega^{-1} \mbox{\quad and\quad } f_6=0.
\]
In terms of the toric coordinate of $X(\delta ^{(n)})$, the image point $\tilde q_1$ has
coordinate
\[
\left(f_0,\frac{\xi _1}{\xi _2},\frac{\xi _2}{\xi _3},\frac{\xi _3}{\xi _4},\frac{\xi _4}{\xi _5},\frac{\xi _5}{\xi _6},
\frac{\xi _6}{\xi _7},\frac{\xi _7}{\xi _8},\frac{\xi _8}{\xi _9},f_9\right) =
\left(0,\omega^{-1},\omega^{-1}, f_3, \omega^{-1},\omega^{-1}, 0,
\omega^{-1},\omega^{-1},0\right).
\]
A permutation of $\{\xi _1,\xi _2,\xi _3\}$ that preserves the ratio $\xi _i/\xi _{i+1}$ is
either of cyclic permutations $(1\; 2\; 3)$ or $(1\; 3\; 2)$.
We have the same thing for $\{\xi _4,\xi _5,\xi _6\}$ and $\{\xi _7,\xi _8,\xi _9\}$.
The permutations $(1\; 2\; 3)$ alone does not fix the point $\tilde q$ because
\[
(1\; 2\; 3)\cdot \frac{\xi _3}{\xi _4}
=\frac{\xi _1}{\xi _4} = \frac{\xi _1}{\xi _2}\frac{\xi _2}{\xi _3}\frac{\xi _3}{\xi _4}
=\omega ^{-2} \frac{\xi _3}{\xi _4}= \omega f_3
\]
and $f_3$ is non-zero as $t_4\neq 0$. More generalliy, as we have
\begin{equation}\label{transformation at joint}
(1\; 2\; 3)^a(4\; 5\; 6)^b \cdot \frac{\xi _3}{\xi _4}
= \omega ^{a-b} \frac{\xi _3}{\xi _4}
\end{equation}
the permutation is in $\Stab _{\mathfrak S_9}(\tilde q)$ only if $a-b=0$.
On the other hand, $(7\; 8\; 9)$ fixes the point $\tilde q$
as $\displaystyle \frac{\xi _6}{\xi _7}=0$. Therefore, we know that
\[
G(\tilde q) = \Stab _{\mathfrak S_9}(\tilde q)=
\langle (1\; 2\; 3)(4\; 5\; 6),\, (7\; 8\; 9)\rangle
\cong (\mathbb Z/3\mathbb Z)^2 \cong \Stab _{G[n]}(\gamma)
\]
as stated in the theorem (note that $\Stab _{\mathfrak S_9}^0(\tilde q)=\{\id \}$ in
this case).

\noindent
(2) Next we consider the case $n=6$ and
$\gamma =\sum _{i=1}^6 (p_i,p'_i)\in
\Sym ^6(S[6]/\mathbb A ^{7})^{ss}$ with $I_{\mathbf t(\gamma)}=\{1<7\}$.
We assume further
\[
(p_1,p'_1)=(p_2,p'_2),\; (p_3,p'_3)=(p_4,p'_4),\;
(p_5,p'_5)=(p_6,p'_6)
\]
so that $\gamma = 2(p_1,p'_1)+2(p_3,p'_3)+2(p_5,p'_5)$,
and for a primitive third root of unity $\omega$
\[
(p_3,p'_3)=\omega\cdot (p_1,p'_1),\quad
(p_5,p'_5)=\omega^2\cdot (p_1,p'_1).
\]
In this case, the coordinate of $\tilde q_1$ is given by
\[
\left(f_0,\frac{\xi _1}{\xi _2},\frac{\xi _2}{\xi _3},\frac{\xi _3}{\xi _4},
\frac{\xi _4}{\xi _5},\frac{\xi _5}{\xi _6},f_6\right)
=(0,1,\omega ^{-1},1,\omega ^{-1},1,0).
\]
and
\[
\Stab _{\mathfrak S_6}^0(\tilde q)
=\mathfrak S\{1,2\}\times\mathfrak S\{3,4\}\times\mathfrak S\{5,6\}.
\]
Therefore a cyclic permutation $(1\; 3\; 5\; 2\; 4\; 6)$
gives an element of $\Stab _{\mathfrak S_6}(\tilde q)$ and
its residue class generates $G(\tilde q)$. Thus we know
$\Stab _{G[n]}(\gamma)\cong G(\tilde q)\cong \mathbb Z/3\mathbb Z$.
\end{example}

\begin{proof}[Proof of Theorem \ref{thm: comparison of stabilizers}]
We keep all the assumptions and the notation above.
The equality $\Stab _{\mathfrak S_n}(\tilde \gamma)=\mathfrak S_{M(\tilde \gamma)}$ for the
partition $\{M(\tilde \gamma)_k\}$ in Lemma \ref{comparison of stab of trivial angle type}
implies that
\[
(p_{m_k},p'_{m_k})=(p_{m_k+1},p'_{m_k+1})=\dots
=(p_{m_{k+1}-1},p'_{m_{k+1}-1})\neq (p_{m_{k+1}},p'_{m_k+1})
\]
and the corresponding cycle $\gamma \in \Sym ^n(S[n]/\mathbb A^{n+1})^{ss}$ is of the form
\[
\gamma =\sum _k \mu_k (p_{m_k},p'_{m_k}),
\]
where $\mu _k = m _{k+1} - m_k$. The restriction $\gamma _l$ of $\gamma$ to
$\Delta ^{i_l,\circ}\times \mathbb A^1$ is
\[
\gamma _l = \sum _{j=0}^{\beta _l-1} \mu _{k_l+j}(p_{m_{k_l+j}},p'_{m_{k_l+j}}).
\]
Now we assume
\[
\mu _{k_l}\geqslant \mu _{k_l+1}\geqslant \dots \geqslant \mu _{k_l+\beta_l-1}
\]
by a further renumbering. Then there is a partition of $\beta _l$
\[
0=\beta _{l,0}<\beta _{l,1}<\cdots <\beta _{l,c_l-1}<\beta _l
\]
and
\[
d_{l,0}>d_{l,1}>\dots >d_{l,c_l-1}>0
\]
such that
\[
d_{l,i}=\mu _{k_l+\beta _{l,i}} = \dots = \mu _{k_l+\beta _{l,i+1}-1}.
\]
Being a finite subgroup of $\mathbb C^*$, $\Stab (\gamma _l)$ is a cyclic group of finite order
consisting of roots of unity. Let $\tau _l\in \mathbb C^*$ be a generator and $r_l$ the order
of $\tau _l$. The action of $\tau _l$ induces a cyclic permutation among the set of points
\[
\{(p_{m_{k_l+j}},p'_{m_{k_l+j}})\, | \, \beta _{l,i}\leqslant j<\beta _{l,i+1}\}
\]
and decomposes the set into a disjoint union of orbits each of which consists of $r_l$ points.
In particular, we have $\beta _{l,i+1}-\beta _{l,i}=r_l\cdot \beta'_{l,i}$ for
some positive integer $\beta'_{l,i}$. We may assume for $0\leqslant \kappa <\beta _{l_i}$
and $0\leqslant j< r_l$,
\[
(p_{m_{k_l+\beta _{l,i}+\kappa r_l+j}},p'_{m_{k_l+\beta _{l,i}+\kappa r_l+j}})
= \tau _l^j (p_{m_{k_l+\beta _{l,i}+\kappa r_l}},p'_{m_{k_l+\beta _{l,i}+\kappa r_l}}).
\]
Summerizing, we configured the index set $K_l^{d_{l,i}}$ as the following:
\[
\xymatrix @R=3pt @C=0pt {
m_{k_l+\beta _{l,i}+\kappa r_l} \ar@/_50pt/[d]_{\tau _l}
  \ar@{<->}@/^20pt/[rrr]^{\mathfrak S_n\mbox{\scriptsize -action}}
  & m_{k_l+\beta _{l,i}+\kappa r_l}+1 & \cdots & m_{k_l+\beta _{l,i}+\kappa r_l}+d_{l,i}-1 \\
m_{k_l+\beta _{l,i}+\kappa r_l+1} \ar@/_50pt/[d]_{\tau _l}
  & m_{k_l+\beta _{l,i}+\kappa r_l+1}+1 & \cdots & m_{k_l+\beta _{l,i}+\kappa r_l+1} +d_{l,i}-1 \\
\qquad \vdots \qquad \ar@/_50pt/[d]_{\tau _l} &\qquad \vdots \qquad & &\qquad \vdots \qquad\\
m_{k_l+\beta _{l,i}+(\kappa +1)r_l-1} & m_{k_l+\beta _{l,i}+(\kappa +1)r_l-1}+1
  & \cdots & m_{k_l+\beta _{l,i}+(\kappa +1)r_l-1} +d_{l,i}-1 \\
}
\]
As the invariant function $\displaystyle \frac{\xi _k}{\xi _{k+1}}=f_k$ is the ratio of the
$u_k/v_k$-coordinates of $p_k$ and $p_{k+1}$ by \eqref{local quotient map},
we get
\[
\begin{pmatrix}
f_{m_{k_l+\beta _{l,i}+\kappa r_l}} & \cdots
& f_{m_{k_l+\beta _{l,i}+\kappa r_l}+d_{l,i}-1} \\
f_{m_{k_l+\beta _{l,i}+\kappa r_l+1}} & \cdots & f_{m_{k_l+\beta _{l,i}+\kappa r_l+1} +d_{l,i}-1}\\
&\vdots \\
f_{m_{k_l+\beta _{l,i}+(\kappa +1)r_l-1}} & \cdots & f_{m_{k_l+\beta _{l,i}+(\kappa +1)r_l-1} +d_{l,i}-1}
\end{pmatrix}
=
\begin{pmatrix}
  1 & 1 & \cdots & 1 & \tau _l^{-1} \\
  1 & 1 & \cdots & 1 & \tau _l^{-1} \\
  && \vdots \\
  1 & 1 & \cdots & 1 & \tau _l^{-1}
\end{pmatrix}
\]
Therefore, a cyclic permutation $c_{m_{k_l},i,\kappa}$ of length $r_l\cdot d_{l,i}$ ``along the column''
\small
\[
\xymatrix @R=3pt @C=5pt {
m_{k_l+\beta _{l,i}+\kappa r_l} \ar@/_40pt/[d]
  &
  & m_{k_l+\beta _{l,i}+\kappa r_l}+1 \ar@/_50pt/[d]
  &
  & \cdots
  &
  & m_{k_l+\beta _{l,i}+\kappa r_l}+d_{l,i}-1 \ar@/_60pt/[d] \\
m_{k_l+\beta _{l,i}+\kappa r_l+1} \ar@/_40pt/[d]
  &
  & m_{k_l+\beta _{l,i}+\kappa r_l+1}+1 \ar@/_50pt/[d]
  &
  & \cdots
  &
  & m_{k_l+\beta _{l,i}+\kappa r_l+1} +d_{l,i}-1 \ar@/_60pt/[d]\\
\qquad\vdots\qquad \ar@/_40pt/[d]
  &
  & \qquad\vdots\qquad \ar@/_50pt/[d]
  &
  &
  &
  & \qquad\vdots\qquad \ar@/_60pt/[d] \\
m_{k_l+\beta _{l,i}+(\kappa +1)r_l-1} \ar `r[ru] `[uuur] [uuurr]
  &
  & m_{k_l+\beta _{l,i}+(\kappa +1)r_l-1}+1 \ar `r[ru] [uuur]
  &
  & \cdots \ar `r[ru] `[uuur] [uuurr]
  &
  & m_{k_l+\beta _{l,i}+(\kappa +1)r_l-1} +d_{l,i}-1
}
\]
\normalsize \baselineskip 17pt \parskip 5pt
fixes this block of coordinates. Therefore, the correspondence
\[
\prod _l \tau _l \mapsto
\left[\prod _l \prod _{i,\kappa} c_{m_{k_l},i,\kappa}\right]
\]
gives a group homomorphism $\Stab _{G[n]}(\gamma)\to
\prod _l \mathfrak S_{K^*_l}$ whose image is contained in
$\rho (G(\tilde q))$. On the other hand,
as the invariant functions $f_k$ determines the relative position of
all $p_k$'s on a component $\Delta ^{i_l,\circ}$,
a permutation in $\mathfrak S_{K^*_l}$ that fixes the point $\tilde q$
is necessarily a power of $[\prod _{i,\kappa} c_{m_{k_l},i,\kappa}]$,
as we saw in the previous example, therefore we get an isomorphims
$\Stab _{G[n]}(\gamma)\cong G(\tilde q)$.
\end{proof}

\begin{proof}[Conclusion of the proof of Theorem \ref{comparison theorem}]
Let us keep our assumptions on
\[
\gamma =\sum (p_i,p'_i) =\sum _k \mu _k (p_{m_k},p'_{m_k})\in \Sym ^n(S[n]/\mathbb A^{n+1})^{ss},
\]
and its lifting $\tilde \gamma\in W[n]^{ss}\times \mathbb A^n$.
Let us recall that the slice $\tilde V_{\tilde \gamma}$ is cut out by the equations
$w_j = w_j(\tilde \gamma _1)$.
More precisely, if we define a non-zero constant $c_j$ by
$c_j = w_j(\tilde \gamma _1)$, we have
\[
 t_j = \frac{v_{j,{j-1}}}{u_{j,{j-1}}}\frac{u_{jj}}{v_{jj}}= w_{n+j} \cdot c_j,
\]
This means that on $\tilde V_{\tilde \gamma}$, a variation of a point $(p_j,p'_j)$
is parametrized by the coordinate $t_j$ and the $j$-th coordinate of the factor
$\mathbb A^n$ of $W[n]^{ss}\times \mathbb A^n$.
Therefore, we know that $V_{\gamma}=\tilde V_{\tilde \gamma}/\Stab _{\mathfrak S_n}(\tilde \gamma)$
is locally isomoprhic to
\[
\left(\prod _k \Sym ^{\mu _k}(\mathbb A^2)\right) \times \mathbb A^1
\subset \Sym ^n(\mathbb A^2)\times \mathbb A^1.
\]
where the last factor $\mathbb A^1$ is the line with coordinate $t_{n+1}$.
Moreover the fiber product
\[
\hat V_{\gamma} =
\Hilb ^n(S[n]/\mathbb A^{n+1})\times_{\Sym ^n(S[n]/\mathbb A^{n+1})} V_{\gamma}\to V_{\gamma}
\]
is isomorphic to a restriction of the Hilbert-Chow morphism
$\Hilb ^n(\mathbb A^2)\times \mathbb A^1\to
\Sym ^n(\mathbb A^2)\times \mathbb A^1$ to $V_\gamma$.
By the universality of Hilbert scheme, the $G(\tilde q)$-equivariant
isomorphism $V_{\gamma}\to U_{\tilde q}$ lifts to an equivariant isomorphism
\[
\hat V_{\gamma} \overset{\sim}{\lto} \hat U_{\tilde q}.
\]
As $\hat V_{\gamma}$ gives a slice to a quotient $\Hilb ^n(S[n]/\mathbb A^{n+1})^{ss}/G[n]$,
this implies that we have an isomorphism of smooth Deligne-Mumford stacks
\[
\mathscr I^n_{S/C}\overset{\sim}{\lto} \mathscr Y^{(n)},
\]
over $Z^{(n)}\cong \Sym ^n(S[n]/\mathbb A^{n+1})^{ss}\git G[n]$,
which completes the proof of Theorem \ref{comparison theorem}.
\end{proof}

\paragraph{Acknowledgement}
The author would like to thank Dr. Ren\'e Birkner, the author of
a Macaulay2 package ``Polyhedra'' \cite{B}.
It would have been impossible for him to find many key observations
in this work without experimentats on the computer program.


\end{document}